\documentclass[11 pt]{article}

\usepackage{amssymb} 
\usepackage{amsmath} 
\usepackage{latexsym}
\usepackage{graphicx}

\usepackage{enumitem}

\graphicspath{ {images/} }
\usepackage{mathtools}

\providecommand{\keywords}[1]{\textbf{\textit{Keywords:}} #1}

\providecommand{\msc}[1]{{\textit{2000 Mathematics Subject Classification:}} #1}
\usepackage{tikz}
\usetikzlibrary{calc}
\usepackage{pgf,tikz}
\usetikzlibrary{arrows}
\usepackage{amscd}
\usepackage{relsize}

\usepackage{tikz-cd}

\definecolor{NoteColor}{rgb}{1,0,0}



\usepackage{footnote}

\usepackage{ color, amsmath, amssymb, amsfonts, amscd, graphicx,latexsym, hyperref}
\usepackage{caption, subcaption}
\usepackage[pdf]{pstricks}
\usepackage{makeidx}

\usepackage[affil-it]{authblk}

\usepackage{amsthm}

\newtheorem{theorem}{Theorem}

\newtheorem{proposition}{Proposition}

\newtheorem{remark}{Remark}

\newtheorem{lemma}{Lemma}

\newtheorem{definition}{Definition}

\newcommand{\R}{\mathbb R}

\title{Minimal stretch maps between Euclidean triangles}
\author{\.{I}smail Sa\u{g}lam  \thanks{Electronic address: \texttt{isaglamtrfr@gmail.com}  }}
\affil{Adana Alparslan Turkes Science and Technology University, \\ Institut de Recherche Mathematique Avanc\'ee, CNRS et Universit\'e de Strasbourg}
\author{Athanase Papadopoulos \thanks{Electronic address: \texttt{papadop@math.unistra.fr}  }}
\affil{ Institut de Recherche Math\'ematique Avanc\'ee, CNRS et Universit\'e de Strasbourg}

\date{}

\begin{document}

\maketitle

\begin{abstract}
Given two triangles whose angles are all acute, we find a homeomorphism with the smallest 
Lipschitz constant between them and we give a formula for the Lipschitz constant of this map. We show that on the set of pairs of acute triangles with fixed area, the function which assigns the logarithm of the smallest Lipschitz constant of Lipschitz maps between them
 is a symmetric 
metric. We show that this metric is Finsler, we give a necessary and  sufficient 
condition for a path in this metric space to be geodesic and we determine the isometry group of this metric space. 

This study is motivated by Thurston's asymmetric metric on the Teichm\"uller space of a hyperbolic surface, and the results in this paper constitute an analysis of a basic Euclidean analogue of Thurston's hyperbolic theory. Many interesting questions in the Euclidean setting deserve further attention.
\end{abstract}
\\
\msc{51F99, 57M50,  32G15, 57K20}


\section{Introduction}

In this paper, we study a metric on a moduli space of Euclidean triangles which is an analogue of Thurston's (asymmetric) metric on the Teichm\"uller space of a surface of finite type.  Thurston introduced his metric in the 1985 preprint \cite{thurston}. Since then, the metric has been studied from various viewpoints. 
 A first survey appeared in 2007 \cite{PT}, in which properties of this metric were compared to analogous properties of the Teichm\"uller metric.  A more recent survey on this metric is in press \cite{Course}.   The paper \cite{S}, published in 2015, contains a set of open problems on this metric. The paper \cite{PS2015} contains several results on the comparison of the Thurston metric with the Teichm\"uller metric. 
Several new techniques have been introduced recently in the study of Thurston's metric; see in particular \cite{AD, HP, LRT}, and the metric has been generalized to various settings, see \cite{DGK, GK} for an analogous metric on spaces of geometrically finite hyperbolic manifolds and \cite{HS} for a generalization in the setting of higher Teichm\"uller theory. Euclidean analogues of this metric have also been studied, see \cite{athanase} for an analogue on the Teichm\"uller space of Euclidean tori, \cite{OMP} for a recent sequel, and the recent work \cite{Wolenski} on an analogue on the moduli space of semi-translation surfaces. 

There is a natural analogue of Thurston's metric on a basic model space, namely, the moduli space of Euclidean triangles. This elementary setting has not been investigated yet. Our aim in this article is to settle the case of the moduli space of acute Euclidean triangles (that is, triangles whose three angles are acute), which turns out to be a natural space to study.

We now present the main results of this paper.

Consider a triangle in the Euclidean plane. Label its vertices by the set $\{v_1,v_2,v_3\}$ so that this labeling induces a counter-clockwise orientation on the boundary of the triangle. We call such a triangle {\it labeled}. 

Consider two labeled triangles $T$ and $T'$ and let $f:  T\to T'$ be a label-preserving homeomorphism. The \emph{Lipschitz constant of $f$} is defined as

$$L(f)=\sup_{x,y \in T, x\neq y}\frac{d_{euc}(f(x),f(y))}{d_{euc}(x,y)}$$

\noindent where $d_{euc}$ is the metric in the Euclidean plane. Now let

\begin{align*}
L(T,T')=&\log\big(\inf\{L(f): f \ \text{is a label-preserving} \\
& \text{homeomorphism between $T$ and $T'$}\}\big). 
 \end{align*}
 
 This formula induces a distance function on the space of Euclidean triangles  which is an analogue of Thurston's Lipschitz metric defined in the hyperbolic setting (see  \cite[p. 4]{thurston} where this distance function is also denoted by $L$).

We obtain the following: 

\begin{itemize}
	\item 
	 In the case where $T$ and $T'$ are acute triangles, we give another formula for the distance $L(T,T')$, which we denote by $m(T,T')$, in terms of  the lengths of the edges and altitudes of the triangles $T$ and $T'$. The new formula   is an analogue of Thurston's definition of his distance in terms of lengths of simple closed geodesics (see  \cite[p. 4]{thurston} where this distance function is denoted by $K$), and the equality $L(T,T')=m(T,T')$ is an analogue of Thurston's equality between his two distance functions $K$ and $L$ (see \cite[p. 40]{thurston}).
	\item
	For every $A>0$, the metric induced by $L$ on the space $\frak{AT}_A$ of acute triangles having fixed area $A$ is Finsler (this is an analogue of Thurston's result in \cite[p. 20]{thurston}).
	\item
	We give a characterization of geodesics  in $\frak{AT}_A$:  a path is geodesic if and only if the angle at each labeled vertex of a triangle in this family of triangles varies monotonically.
	\item
	The isometry group of $\frak{AT}_A$ is isomorphic to $S_3$, the group of permutations of $\{1,2,3\}$.
	
\end{itemize} 

The paper is organized as follows. Section \ref{section-acute} is the technical heart of our work. For any two labeled triangles $T$ and $T'$, we define their Lipschitz distance, and then we define the distance $m(T,T')$ and show that $m(T,T')=L(T,T')$ for any acute triangles $T$ and $T'$. 
In Section \ref{section-space-triangles} we introduce several spaces of triangles and study some of their topological and metric properties. Among these spaces, the space of acute triangles and the space of non-obtuse triangles having area $A$ will play central roles in this paper. We denote these spaces by $\frak{AT}_A$ and $\overline{\frak{AT}_A}$ respectively.
In Section \ref{section-special-geodesic} we give a necessary and sufficient condition for a path in $\overline{\frak{AT}_A}$ to be a geodesic. In Section \ref{section-finsler} we prove that  the metric $L$ on $\frak{AT}_A$ is Finsler. We determine the isometry group 
of $\overline{\frak{AT}_A}$ in Section \ref{section-symmetry}.

We introduce some notation used throughout the  paper. Let $T$ be a labeled triangle, with vertices $v_1, v_2, v_3$.  For $i \in \{1,2,3\}$, we denote the edge opposite to the vertex $v_i$ by $e_i$. We denote the angle at the vertex $v_i$ by $\theta_i$ and the altitude from the vertex $v_i$ by $h_i$.
We let $\lvert\lvert e_i \rvert \rvert $ and $\lvert \lvert h_i \rvert \rvert$ be the lengths of $e_i$ and $h_i$, respectively. We denote the intersection of the altitude $h_i$
with the line which contains the edge $e_i$ by $p_i$. A triangle with vertices $v_1,v_2,v_3$ is denoted  by $\Delta v_1v_2v_3$. The line segment between two points $x$ and $y$ in the Euclidean plane is denoted by $[x,y]$, and its length by $\lvert[x,y]\rvert$. Finally, $\mathrm{Area}(T)$ will denote the area of a triangle $T$. 

 \section{Acute triangles}
\label{section-acute}

Assume that $T$ and $T'$ are two acute triangles. We use the notation introduced above for the triangle $T$. Similarly, for  $T'$, we denote the  vertices by $v'_1, v'_2, v'_3$, the edge opposite to the end $v'_i$ by $e'_i$,  the angle at the vertex $e'_i$ by $\theta'_i$, etc.

 \begin{proposition}
 	\label{esitsizlik}
 For any two acute triangles $T$ and $T'$, we have:
$$\exp(L(T,T'))\geq \max\{\frac{\lvert\lvert e'_1 \rvert \rvert}{\lvert \lvert e_1 \rvert\rvert}, \frac{\lvert\lvert e'_2 \rvert \rvert}{\lvert \lvert e_2 \rvert\rvert},\frac{\lvert\lvert e'_3 \rvert \rvert}{\lvert \lvert e_3 \rvert\rvert},\frac{\lvert\lvert h'_1 \rvert \rvert}{\lvert \lvert h_1 \rvert\rvert},\frac{\lvert\lvert h'_2 \rvert \rvert}{\lvert \lvert h_2 \rvert\rvert},\frac{\lvert\lvert h'_3 \rvert \rvert}{\lvert \lvert h_3 \rvert\rvert}\}.$$
\end{proposition}

\begin{proof}
 
Since any label-preserving homeomorphism from $T$ to $T'$ sends $e_1$ to $e'_1$, $e_2$ to $e'_2$ and $e_3$ to $e'_3$, it is clear that 

$$\exp(L(T,T'))\geq \max\{ \frac{\lvert\lvert e'_1 \rvert \rvert}{\lvert \lvert e_1 \rvert\rvert}, \frac{\lvert\lvert e'_2 \rvert \rvert}{\lvert \lvert e_2 \rvert\rvert},\frac{\lvert\lvert e'_3 \rvert \rvert}{\lvert \lvert e_3 \rvert\rvert} \}.$$

Let $f: T\to T'$ be a label-preserving homeomorphism. Since $T$ is acute it follows that each $p_i$ lies in the interior of the edge $e_i$. Therefore we have 

$$d_{euc}(f(p_i),f(v_i))=d_{euc}(f(p_i),v'_i)\geq \lvert\lvert h'_i \rvert \rvert.$$

Since 

$$d_{euc}(p_i,v_i)=\lvert\lvert h_i\rvert\rvert,$$
 it follows that $L(f)\geq \frac{\lvert\lvert h'_i\rvert\rvert}{\lvert\lvert h_i\rvert\rvert}$. Therefore for any two acute triangles $T$ and $T'$, we have:

 $$\exp(L(T,T'))\geq \max\{\frac{\lvert\lvert e'_1 \rvert \rvert}{\lvert \lvert e_1 \rvert\rvert}, \frac{\lvert\lvert e'_2 \rvert \rvert}{\lvert \lvert e_2 \rvert\rvert},\frac{\lvert\lvert e'_3 \rvert \rvert}{\lvert \lvert e_3 \rvert\rvert},\frac{\lvert\lvert h'_1 \rvert \rvert}{\lvert \lvert h_1 \rvert\rvert},\frac{\lvert\lvert h'_2 \rvert \rvert}{\lvert \lvert h_2 \rvert\rvert},\frac{\lvert\lvert h'_3 \rvert \rvert}{\lvert \lvert h_3 \rvert\rvert}\},$$
 which is the  inequality we need.
 \end{proof}

For two arbitrary labeled triangles $T$ and $T'$ we define
\begin{equation}\label{formula-m}
 m(T,T')=\log(\max\{\frac{\lvert\lvert e'_1 \rvert \rvert}{\lvert \lvert e_1 \rvert\rvert}, \frac{\lvert\lvert e'_2 \rvert \rvert}{\lvert \lvert e_2 \rvert\rvert},\frac{\lvert\lvert e'_3 \rvert \rvert}{\lvert \lvert e_3 \rvert\rvert},\frac{\lvert\lvert h'_1 \rvert \rvert}{\lvert \lvert h_1 \rvert\rvert},\frac{\lvert\lvert h'_2 \rvert \rvert}{\lvert \lvert h_2 \rvert\rvert},\frac{\lvert\lvert h'_3 \rvert \rvert}{\lvert \lvert h_3 \rvert\rvert}\}).
\end{equation}

\noindent Assume that we scale the triangle $T$ by a factor $\lambda$ and the triangle $T'$ by a factor $\lambda'$, where $\lambda, \lambda'>0$. Let us denote the scaled triangles by $\lambda T$ and $\lambda' T'$. This means that the triangle $\lambda T$ has edge lengths $\lambda \lvert\lvert e_1\rvert \rvert$, $\lambda \lvert\lvert e_2\rvert \rvert$ and $\lambda \lvert\lvert e_3\rvert \rvert$ and that the triangle $\lambda'T'$
has side lengths $\lambda' \lvert\lvert e'_1\rvert \rvert$, $\lambda' \lvert\lvert e'_2\rvert \rvert$ and $\lambda' \lvert\lvert e'_3\rvert \rvert$. The following formulae are clear:
\begin{equation}
	\label{scaling-L}
 \exp(L(\lambda T, \lambda' T'))= \frac{\lambda'}{\lambda}\exp(L(T,T'))
  \end{equation}

\begin{equation}
	\label{scaling-m}
\exp(m(\lambda T, \lambda' T'))= \frac{\lambda'}{\lambda} \exp(m(T,T'))
\end{equation}

\begin{remark}
	We shall prove that for any two acute triangles $T$ and $T'$, $L(T,T')=m(T,T')$. See Theorem \ref{acute-proof}. We shall use the following fact in the proof. If $\lambda,\lambda'>0$, $L(T,T')=m(T,T')$ if and only if $L(\lambda T,\lambda' T')=m(\lambda T, \lambda' T)$. This follows from Inequalities (\ref{scaling-L}) and (\ref{scaling-m}).
\end{remark}

\subsection{Right triangles}
\label{right-triangles}

Right triangles appear naturally as sitting on the boundary of the moduli space of acute triangle. 

Assume that $T$ and $T'$ are two right triangles where $\theta_1$ and $\theta'_1$ are equal to $\frac{\pi}{2}$. See Figure \ref{right-triangle}. In this subsection we calculate $L(T,T')$.

\begin{figure}
	\hspace*{-4cm} 
	\includegraphics[scale=0.7]{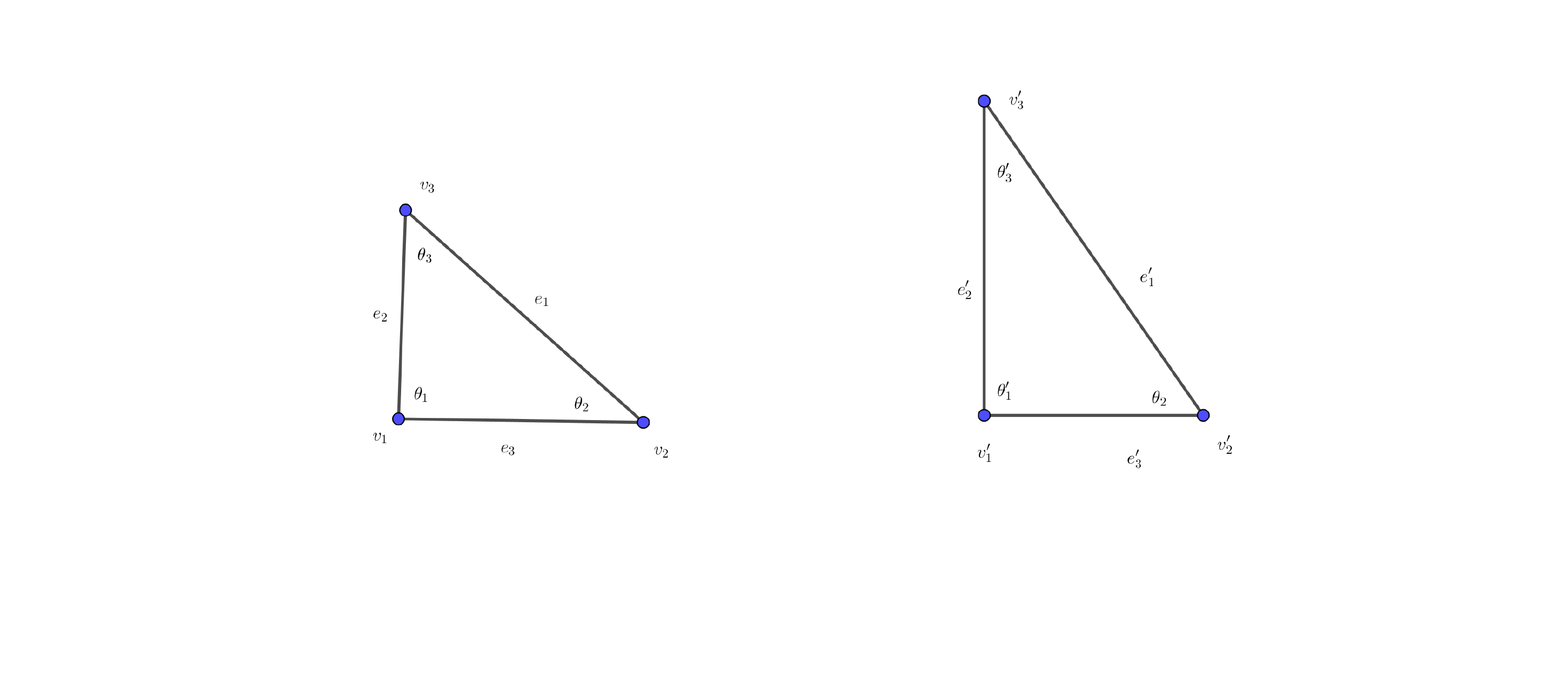}
	
	\caption{A best Lipschitz map between the two triangles is the affine map.}
	\label{right-triangle}
\end{figure}

\begin{proposition}
	Let $T$ and $T'$ be two right triangles so that $\theta_1=\theta'_1=\frac{\pi}{2}$. Then $$\exp(L(T,T'))= \max\{ \frac{\lvert\lvert e'_3 \rvert \rvert}{\lvert \lvert e_3 \rvert\rvert}, \frac{\lvert \lvert e'_2\rvert \rvert}{\lvert \lvert e_2\rvert \rvert} \}.$$
\end{proposition}

\begin{proof}
\noindent It is clear that 

\begin{equation}\label{eq:geq}
\exp(L(T,T'))\geq \max\{ \frac{\lvert\lvert e'_3 \rvert \rvert}{\lvert \lvert e_3 \rvert\rvert}, \frac{\lvert \lvert e'_2\rvert \rvert}{\lvert \lvert e_2\rvert \rvert} \}.
\end{equation}

\noindent Let us show that the other inequality is true as well.

We consider a Euclidean system of coordinates $(x,y)$ in $\mathbb{R}^2$.

\noindent By performing some isometries to $T$ and $T'$, we may assume that the vertices $v_1$ and $v'_1$ are at the origin, the sides $e_3$ and $e'_3$ are on the $x$-axis, and the sides $e_2$ and $e'_2$ are on the $y$-axis.
Consider the following homeomorphism from $T$ to $T'$:

$$f: (x,y)\to (\frac{\lvert \lvert e'_3\rvert \rvert  }{\lvert \lvert e_3\rvert \rvert}x,\frac{\lvert \lvert e'_2\rvert \rvert  }{\lvert \lvert e_2\rvert \rvert}y).$$

\noindent It is clear that 

$$L(f)\geq \max\{ \frac{\lvert\lvert e'_3 \rvert \rvert}{\lvert \lvert e_3 \rvert\rvert}, \frac{\lvert \lvert e'_2\rvert \rvert}{\lvert \lvert e_2\rvert \rvert} \}.$$

\noindent Let $q_1=(x_1,y_1)$ and $q_2=(x_2,y_2)$ be two points in $T$. 

Then 

$$d_{euc}(f(q_1),f(q_2))= d_{euc}((\frac{\lvert \lvert e'_3\rvert \rvert  }{\lvert \lvert e_3\rvert \rvert}x_1,\frac{\lvert \lvert e'_2\rvert \rvert  }{\lvert \lvert e_2\rvert \rvert}y_1),(\frac{\lvert \lvert e'_3\rvert \rvert  }{\lvert \lvert e_3\rvert \rvert}x_2,\frac{\lvert \lvert e'_2\rvert \rvert  }{\lvert \lvert e_2\rvert \rvert}y_2) )
$$
$$
= \sqrt{(\frac{\lvert \lvert e'_3\rvert \rvert  }{\lvert \lvert e_3\rvert \rvert})^2(x_2-x_1)^2+(\frac{\lvert \lvert e'_2\rvert \rvert  }{\lvert \lvert e_2\rvert \rvert})^2(y_2-y_1)^2}
$$
$$\leq \max\{ \frac{\lvert\lvert e'_3 \rvert \rvert}{\lvert \lvert e_3 \rvert\rvert}, \frac{\lvert \lvert e'_2\rvert \rvert}{\lvert \lvert e_2\rvert \rvert} \}
\sqrt{(x_2-x_1)^2+(y_2-y_1)^2}
$$

$$ = \max\{ \frac{\lvert\lvert e'_3 \rvert \rvert}{\lvert \lvert e_3 \rvert\rvert}, \frac{\lvert \lvert e'_2\rvert \rvert}{\lvert \lvert e_2\rvert \rvert} \} d(q_1,q_2).$$

\noindent Therefore $L(f)\leq \max\{ \frac{\lvert\lvert e'_3 \rvert \rvert}{\lvert \lvert e_3 \rvert\rvert}, \frac{\lvert \lvert e'_2\rvert \rvert}{\lvert \lvert e_2\rvert \rvert} \}$. Combined with (\ref{eq:geq}), this gives $L(f)=\max\{ \frac{\lvert\lvert e'_3 \rvert \rvert}{\lvert \lvert e_3 \rvert\rvert}, \frac{\lvert \lvert e'_2\rvert \rvert}{\lvert \lvert e_2\rvert \rvert} \}$. 
From this we conclude that $\exp(L(T,T'))= \max\{ \frac{\lvert\lvert e'_3 \rvert \rvert}{\lvert \lvert e_3 \rvert\rvert}, \frac{\lvert \lvert e'_2\rvert \rvert}{\lvert \lvert e_2\rvert \rvert} \}$. 
\end{proof}

Note that $\exp(L(T,T'))$ is given by an infimum and this infimum is attained by the map we just constructed. Therefore the map we defined above is a {\it best Lipschitz map}.

\subsection{A problem related to right triangles}
\label{another-problem-right}
\begin{figure}
	\hspace*{-3cm} 
	\includegraphics[scale=0.7]{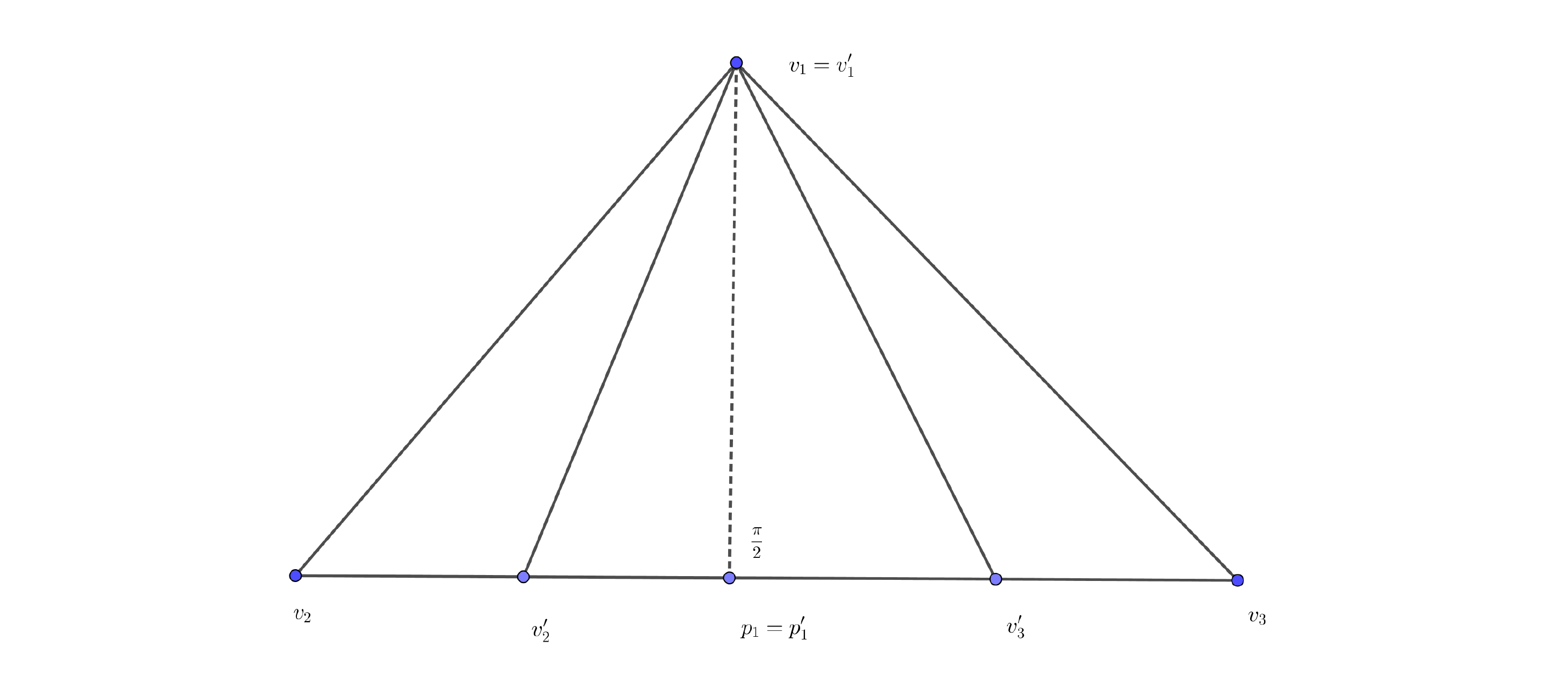}
	
	\caption{If the triangle $T'$ is  contained in the triangle $T'$ and $T$ and $T'$ have the same altitudes, then $\exp(L(T,T'))=1$.}
	\label{another-problem}
\end{figure}

\begin{proposition}
	Let $T$ and $T'$ be two acute triangles so that $\theta_2\leq \theta_2'$ and $\theta_3\leq \theta_3'$. Assume that $\lvert\lvert h_1\rvert \rvert =\lvert \lvert h_1'\rvert\rvert$. Then $\exp(L(T,T'))=1$ and there is a label-preserving homeomorphism $f:T\to T'$ so that $L(f)=1$.
\begin{proof}
We can move the triangle $T$ into $T'$ so that their altitudes coincide. See Figure \ref{another-problem}. Therefore we have two triangles  $T=\Delta v_1v_2v_3$ and $T'=\Delta v'_1v'_2v'_3$ so that $v_1=v'_1$ and $h_1=h'_1$. We want to find $L(T,T')$. Since $T$ and $T'$ have the same altitude from $v_1=v'_1$, we readily see that for any label-preserving homeomorphism $f:T\to T'$, $L(f)\geq1$. Now we construct a label-preserving homeomorphism $f: T\to T'$ so that $L(f)=1$;  this will complete the proofs of the two statements in the proposition.
Consider the homeomorphism $f$ as in Section \ref{right-triangles} which sends the triangle $\Delta v_1v_2p_1$ to the triangle 
$\Delta v_1v'_2p_1$. Also consider the homeomorphism $h$ as in Section \ref{right-triangles} which sends the triangle  $\Delta v_1 p_1 v_3$ to the triangle $\Delta v_1p_1v'_3$. Clearly $L(g)=L(h)=1$ and since $g$
and $h$ agree on the altitude $h_1=h'_1=[v_1,p_1]$, they induce a  homeomorphism $f:T\to T'$. Furthermore it is clear that $L(f)=1$.

\end{proof}

\end{proposition}

\subsection{Acute triangles}

We are now ready to prove the main result on the space of acute triangles. It is an analogue, in the elementary case we are discussing,  of a result of Thurston in \cite{thurston} (see Corollary 8.5).

\begin{theorem}
	\label{acute-proof}
	If $T$ and $T'$ are two acute triangles, then
	
	$$L(T,T')=m(T,T').$$
	Furthermore, there exists a minimal Lipschitz map between the two acute triangles.
	
	\begin{proof}
		As already noted, we have $\exp(L(T,T'))\geq \exp(m(T,T'))$. See Proposition \ref{esitsizlik}. We will show that $\exp(L(T,T'))\leq \exp(m(T,T'))$. There exist distinct  $i,j \in \{1,2,3\}$ such that  either
		\begin{enumerate}
			\item 
			$\theta_i\leq \theta'_i$ and $\theta_j \leq \theta'_j$, or
			\item
			$\theta_i\geq \theta'_i$ and $\theta_j \geq \theta'_j$.
			
		\end{enumerate}
	
Without loss of generality we suppose that 	$\{i,j\}=\{2,3\}$.

	Assume that the first case holds, that is, $\theta_2 \leq \theta'_2$ and $\theta_3\leq \theta'_3$. We may scale the triangles $T$ and $T'$ so that they have altitudes of the same length from the vertex labeled by 1. That is, we may suppose that $\lvert\lvert h_1 \rvert\rvert=\lvert\lvert h_1'\rvert\rvert$. In that case the triangle $T'$ can be moved inside the triangle $T$ so that their altitudes coincide. See Figure \ref{another-problem}.
	Therefore there exists a homeomorphism $f:T\to T'$ such that $L(f)=1$. Hence we have
	
	$$\exp(L(T,T'))\leq L(f)=1\leq\exp( m(T,T')).$$
	
	\noindent This completes the proof of the claim for the first case.

	\begin{figure}
		\hspace*{-4cm} 
		\includegraphics[scale=0.7]{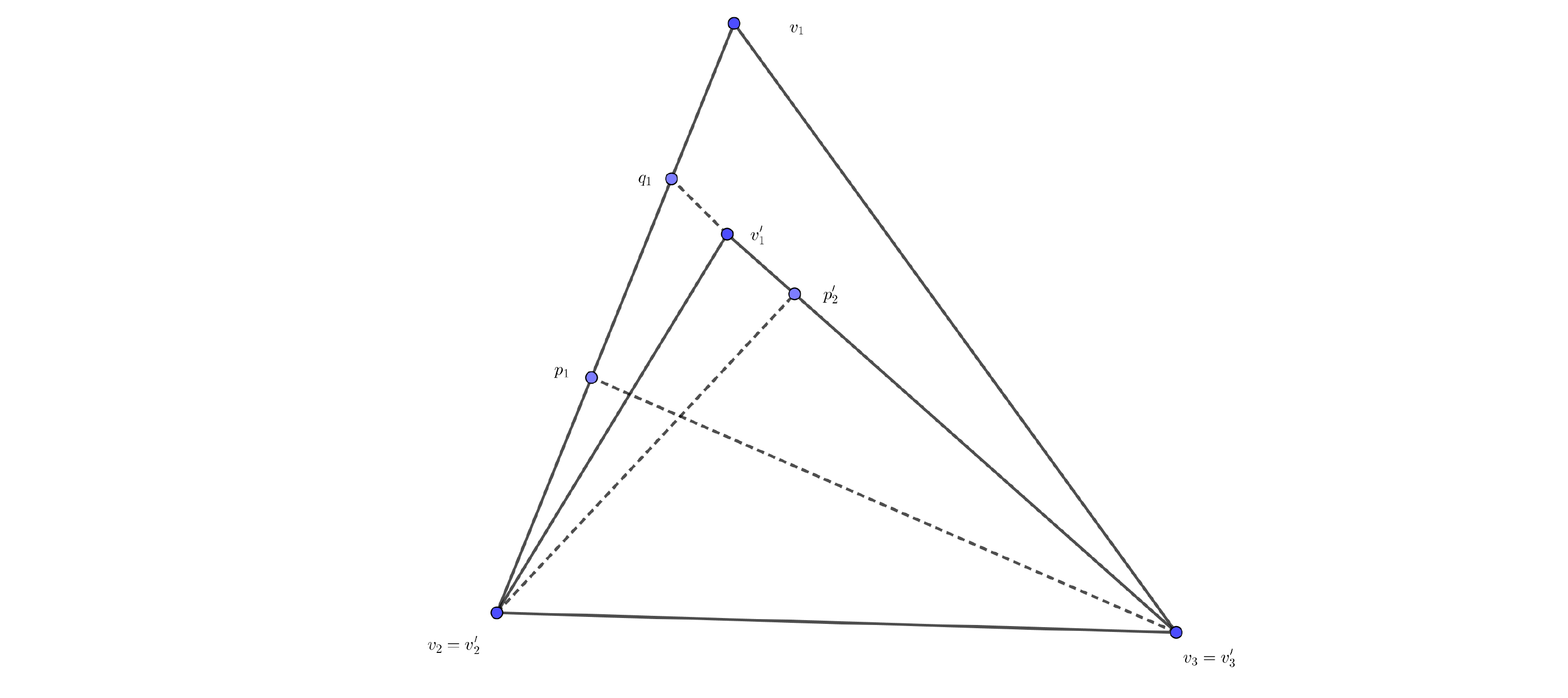}
		
		\caption{The case where $\theta_2\geq \theta'_2$ and $\theta_3\geq \theta'_3$}
		\label{iki}
	\end{figure}

	Assume now that the second case holds, that is, $\theta_2\geq \theta'_2$ and $\theta_3\geq \theta'_3$.  After scaling the triangle $T$ by $\frac{1}{\lvert\lvert e_1\rvert\rvert }$
	and the triangle $T'$ by $\frac{1}{\lvert\lvert e'_1\rvert \rvert}$, we may suppose that the lengths of $e_1$ and $e'_1$ are equal to 1. Also, we may assume that $T$ and $T'$ share an edge and two vertices, that is, $e_1=e'_1$, $v_2=v'_2$ and $v_3=v'_3$.
	It follows that $T$ contains $T'$. See Figure \ref{iki}. Consider the triangle $\Delta q_1v_2v_3$ where $q_1$ is the intersection of the line passing through the points $v'_1$ and $v_3$ with the line segment $[v_1,v_2]$. Consider the homeomorphism $g$ sending $\Delta v_1v_2v_3$ to $\Delta q_1v_2v_3$ which is the identity on the triangle $\Delta p_1v_2v_3$ and which maps the $\Delta v_1p_1v_3$ to the triangle $\Delta q_1p_1v_3$ as in the Section \ref{right-triangles}. Clearly $L(g)=1$. 
	
		Now consider the homeomorphism $h$ from $\Delta q_1v_2v_3$ to $\Delta v'_1v_2v_3$
	which is the identity on the triangle $\Delta p'_2v_2v_3$ and which sends $\Delta q_1v_2p'_2$ to $\Delta v'_1v_2p'_2$ as in Section \ref{right-triangles}. Clearly $L(h)=1$. 
	
Therefore $h\circ g$ is a homeomorphism between $T$ and $T'$ satisfying $L(h\circ g)=1$. We have 

	$$\exp(L(T,T'))\leq L(h\circ g)=1\leq \exp(m(T,T')).$$

	In particular $L(T,T')$ is given by an infimum which is attained by the map constructed 
	in the proof of Theorem \ref{acute-proof} when $T$
	and $T'$ are acute triangles. Thus, there exists a best Lipschitz map between two acute triangles.
\end{proof}
\end{theorem}

\begin{remark}
	\label{equal-area}
Assume that $T$ and $T'$ have the same area. Since 

$$\lvert\lvert e_i\rvert\rvert\cdot \lvert\lvert h_i\rvert\rvert=\lvert\lvert e'_i\rvert\rvert\cdot \lvert\lvert h'_i\rvert\rvert,$$

\noindent it follows that $$\frac{\lvert\lvert e_i\rvert\rvert}{\lvert\lvert e'_i\rvert\rvert}=\frac{\lvert\lvert h'_i\rvert\rvert}{\lvert\lvert h_i\rvert\rvert} \ \text{for any}\ i\in \{1,2,3\}.$$

Therefore $$m(T,T') =\log(\max\{\frac{\lvert\lvert e'_1 \rvert \rvert}{\lvert \lvert e_1 \rvert\rvert}, \frac{\lvert\lvert e'_2 \rvert \rvert}{\lvert \lvert e_2 \rvert\rvert},\frac{\lvert\lvert e'_3 \rvert \rvert}{\lvert \lvert e_3 \rvert\rvert},\frac{\lvert\lvert e_1 \rvert \rvert}{\lvert \lvert e'_1 \rvert\rvert},\frac{\lvert\lvert e_2 \rvert \rvert}{\lvert \lvert e'_2 \rvert\rvert},\frac{\lvert\lvert e_3 \rvert \rvert}{\lvert \lvert e'_3 \rvert\rvert}\})$$
$$=\max_i\{\lvert \log (\lvert\lvert e'_i\rvert\rvert)-\log(\lvert\lvert e_i \rvert\rvert)\rvert\}.$$ 

\end{remark}

\begin{remark} (The case of obtuse triangles.)
Consider Figure \ref{counter-example}. Let $T$ and $T'$ be the triangles $\Delta v_1 v_2 v_3$ and $\Delta v'_1v'_2v'_3$, respectively. Note that $v_2=v'_2$ and $v_3=v'_3$. We will show that $L(T,T')>m(T,T')$. Indeed it is easy to see that 

$$\exp(m(T,T'))=\frac{\lvert \lvert h'_1 \rvert \rvert}{\lvert \lvert h_1 \rvert \rvert}=\frac{3}{1}=3.$$  

Consider the point $p_1$ which is the intersection of the altitude from the vertex $v_1$ and the edge $e_1$. If $f$ is any label-preserving  homeomorphism from $T$ to $T'$. Then $f(p_1)$ is on the interior of the edge $e_1$. It follows that 

$$d_{euc}(f(p_1),f(v_1))=d_{euc}(f(p_1),v'_1)>d_{euc}(v_2,v_1)=2\sqrt{3}.$$
 Hence \begin{equation}
	\label{not-attain}
	L(f)> \frac{d_{euc}(f(p_1),f(v_1))}{d_{euc}(p_1,v_1)}=2\sqrt{3}.
	\end{equation}
Thus it follows that $\exp(L(T,T'))\geq 2\sqrt{3}>3=\exp(m(T,T'))$. Note that one can scale the triangles $T$ and $T'$ to get  $\lambda T$ and $\lambda'T'$ so that they have same area. In that case we  have $L(\lambda T, \lambda'T')>m(\lambda T, \lambda' T')$.
\end{remark}

\begin{figure}
	\hspace*{-1.3cm} 
	\includegraphics[scale=0.92]{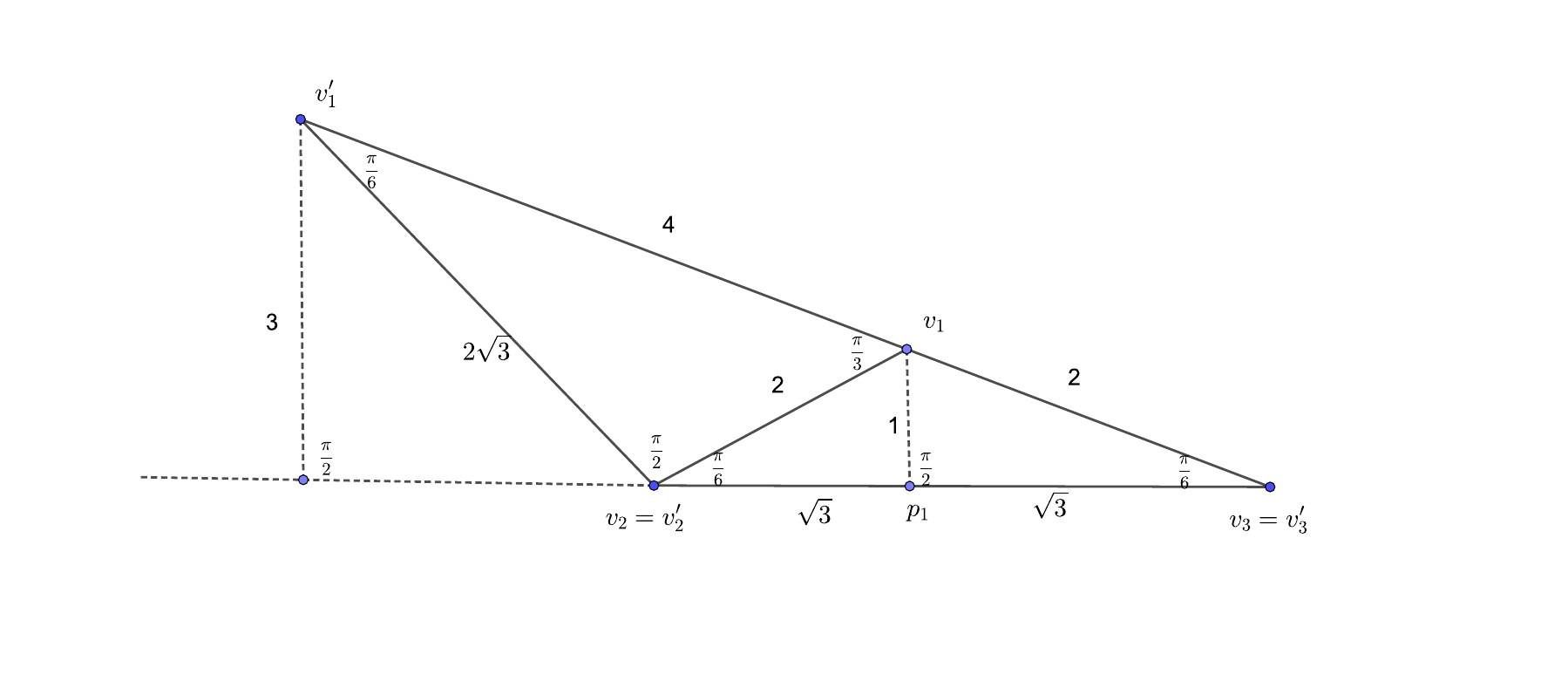}
	
	\caption{Two triangles $T$ and $T'$ so that $L(T,T')>m(T,T')$.}
	\label{counter-example}
\end{figure}

	

\subsection{Some Facts about Two Non-obtuse Triangles Sharing an Edge}

\label{a-remark-about}
Let $T$ be a triangle with vertices $v_1,v_2,v_3$ and $T'$ be a triangle with vertices $v'_1,v'_2,v'_3$ such that $v_1=v'_1$ and $v_2=v'_2$. Suppose that $T$ and $T'$ are non-obtuse and $T$ is contained in $T'$, as in Figure \ref{istek}. 
We claim that 

$$\exp(m(T,T'))=\frac{\lvert\lvert h'_3 \rvert \rvert}{\lvert\lvert h_3 \rvert \rvert}.$$

\begin{figure}
	\hspace*{-4cm} 
	\includegraphics[scale=0.65]{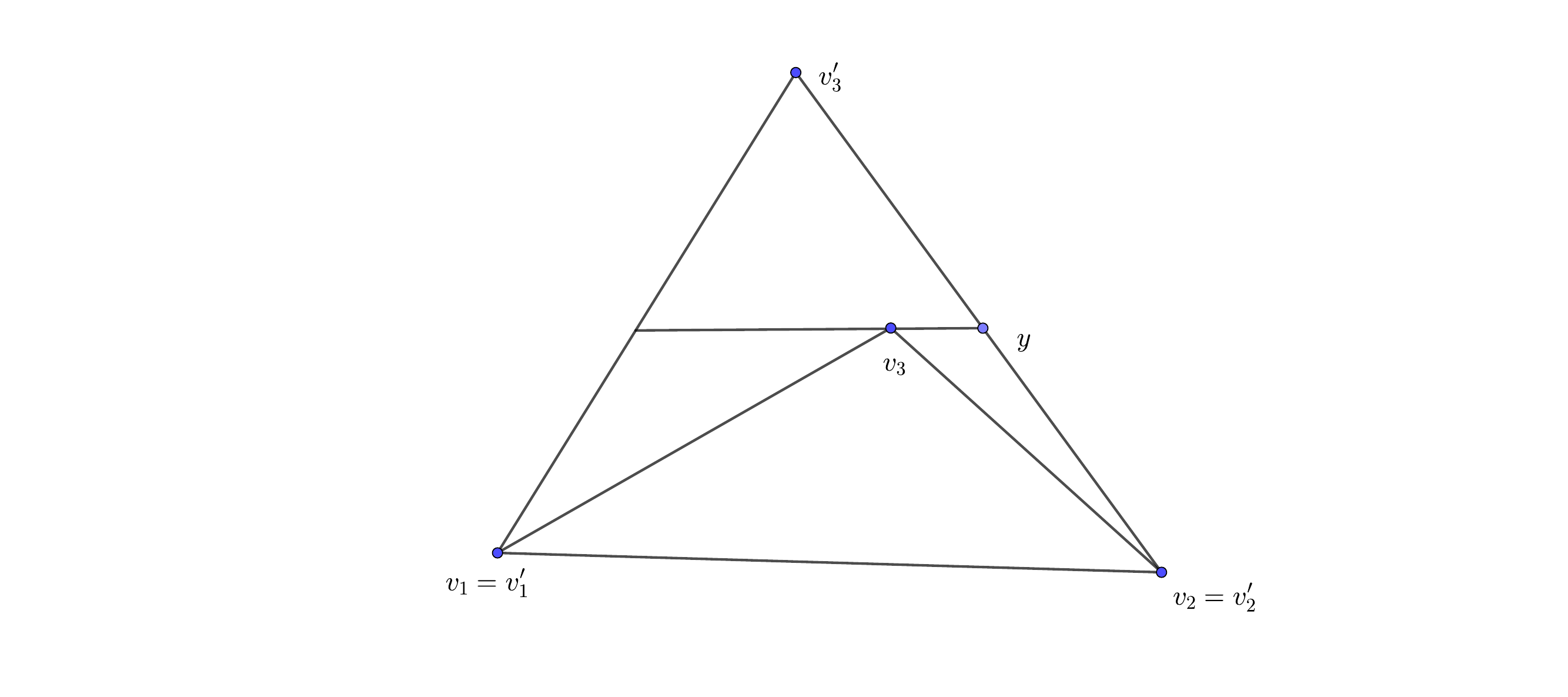}
	
	\caption{}
	\label{istek}
\end{figure}

 Consider the line which passes through $v_3$ and which is parallel to the line passing through $v_1$ and $v_2$. Let $y$ be the point of intersection of this line with the edge $e'_1$. Then we have

$$\frac{\lvert\lvert h'_3 \rvert \rvert}{\lvert\lvert h_3 \rvert \rvert}=\frac{\lvert\lvert e'_1\rvert\rvert}{\lvert\lvert [v_2,y]\rvert\rvert}\geq\frac{\lvert\lvert e'_1\rvert\rvert}{\lvert\lvert e_1\rvert\rvert}.$$

\noindent Similarly 

$$\frac{\lvert\lvert h'_3 \rvert \rvert}{\lvert\lvert h_3 \rvert \rvert}\geq\frac{\lvert\lvert e'_2\rvert\rvert}{\lvert\lvert e_2\rvert\rvert}.$$

\noindent Now we show that 
$$\frac{\lvert\lvert h'_3 \rvert \rvert}{\lvert\lvert h_3 \rvert \rvert}\geq\frac{\lvert\lvert h'_1\rvert\rvert}{\lvert\lvert h_1\rvert\rvert}.$$

\noindent We have 
$$\frac{\lvert\lvert h'_3 \rvert \rvert}{\lvert\lvert h_3 \rvert \rvert}=\frac{\mathrm{Area}(T')}{\mathrm{Area}(T)}=\frac{\lvert\lvert h'_1 \rvert \rvert}{\lvert\lvert h_1 \rvert \rvert}\frac{\lvert\lvert e'_1\rvert\rvert}{\lvert\lvert e_1\rvert\rvert}.$$
\noindent Since $\frac{\lvert\lvert e'_1\rvert\rvert}{\lvert\lvert e_1\rvert\rvert}\geq 1$, we have $\frac{\lvert\lvert h'_3 \rvert \rvert}{\lvert\lvert h_3 \rvert \rvert}\geq\frac{\lvert\lvert h'_1\rvert\rvert}{\lvert\lvert h_1\rvert\rvert}$. Similarly, $\frac{\lvert\lvert h'_3 \rvert \rvert}{\lvert\lvert h_3 \rvert \rvert}\geq\frac{\lvert\lvert h'_2\rvert\rvert}{\lvert\lvert h_2\rvert\rvert}.
$
Since the other arguments of $\exp(m(T,T'))$ are less than or equal to 1, we have 

$$\exp(m(T,T'))=\frac{\lvert\lvert h'_3 \rvert \rvert}{\lvert\lvert h_3 \rvert \rvert}.$$

Now we want to scale $T$ so that the new triangle and $T'$ have the same area. Clearly we need to scale $T$ by the factor $\lambda=\frac{\sqrt{\lvert\lvert h'_3\lvert\lvert}}{\sqrt{\lvert\lvert h_3 \rvert \rvert}}$. Let $T''=\lambda T$. Then Equality \ref{scaling-m} implies that 
$$\exp(m(T'',T'))=\lambda=\frac{\sqrt{\lvert\lvert h'_3\lvert\lvert}}{\sqrt{\lvert\lvert h_3 \rvert \rvert}}=\frac{\lvert\lvert e_3''\rvert\rvert}{\lvert\lvert e_3'\rvert\rvert},$$

\noindent where $e_3''$ is the edge of $T''$ which is opposite the vertex with label 3, $v_3''$. Let us summarize the above discussion as a lemma.

\begin{lemma}
	\label{crucial}
	Let $T$ and $T'$ be labeled non-obtuse triangles having the same area. There exists $i\in \{1,2,3\}$ such that 
	 \begin{enumerate}
	 	\item 
	 	$\theta_j\leq \theta'_j$ and $\theta_k \leq \theta'_k$, or
	 	\item
	 	$\theta_j\geq \theta'_j$ and $\theta_k \geq \theta'_k$,
	 	
	 \end{enumerate}
 \noindent where $j,k\in \{1,2,3\}$, $j,k\neq i$. Then, we have
 $$\exp(m(T,T'))=\max\{\frac{\lvert\lvert e_i\rvert\rvert}{\lvert\lvert e_i'\rvert\rvert},\frac{\lvert\lvert e_i'\rvert\rvert}{\lvert\lvert e_i\lvert\lvert}\},$$
 
\noindent or equivalently,

$$ m(T,T')=\lvert \log(\lvert\lvert e_i'\rvert\rvert)-\log(\lvert\lvert e_i\rvert\rvert)\rvert.$$

\noindent More precisely,

\begin{enumerate}
	\item 
	If 	 	$\theta_j\leq \theta'_j$ and $\theta_k \leq \theta'_k$ then
	\begin{equation}
		\label{normalized1}
	\exp(m(T,T'))=\frac{\lvert\lvert e_i\rvert\rvert}{\lvert\lvert e_i'\rvert\rvert},
	\end{equation}
	\item
	and if $\theta_j\geq \theta'_j$ and $\theta_k \geq \theta'_k$ then
	\begin{equation}
		\label{normalized2}
	\exp(m(T,T'))=\frac{\lvert\lvert e_i'\rvert\rvert}{\lvert\lvert e_i\rvert\rvert}.
	\end{equation}
\end{enumerate}
\end{lemma}

\begin{lemma}\label{lem:two}
	Let $T$ and $T'$ be two labeled triangles of the same area with angles $(\theta_1,\theta_2,\theta_3)$
	and $(\theta_1',\theta_2',\theta_3')$. Let $\{i,j,k\}=\{1,2,3\}$.
	
	\begin{enumerate}
		\item 
	If $\theta_i>\theta_i'$, $\theta_j<\theta_j'$ and $\theta_k\leq \theta_k'$, then
	
	$$\exp(m(T,T'))=\frac{\lvert\lvert e_i\rvert \rvert}{\lvert \lvert e_i'\rvert\rvert}>\max\{\frac{\lvert\lvert e_k\rvert\rvert}{\lvert\lvert e_k'\rvert\rvert},\frac{\lvert\lvert e_k'\rvert\rvert}{\lvert\lvert e_k\rvert\rvert}\}.$$
	\item
	If $\theta_i>\theta_i'$, $\theta_j<\theta_j'$ and $\theta_k>\theta_k'$, then
	
	$$\exp(m(T,T'))=\frac{\lvert\lvert e_j'\rvert \rvert}{\lvert \lvert e_j\rvert\rvert}>\max\{\frac{\lvert\lvert e_k\rvert\rvert}{\lvert\lvert e_k'\rvert\rvert},\frac{\lvert\lvert e_k'\rvert\rvert}{\lvert\lvert e_k\rvert\rvert}\}.$$	
	\end{enumerate}

\begin{proof}Without loss of generality we will assume that $i=1,j=2,k=3$.
	\begin{enumerate}
		\item 	By Lemma \ref{crucial}, we know that $\exp(m(T,T'))=\frac{\lvert\lvert e_1\rvert \rvert}{\lvert \lvert e_1'\rvert\rvert}$, and clearly $\exp(m(T,T'))>1$. This means that 
		$\frac{\lvert\lvert e_1\rvert \rvert}{\lvert \lvert e_1'\rvert\rvert}\geq \max\{\frac{\lvert\lvert e_3\rvert\rvert}{\lvert\lvert e_3'\rvert\rvert},\frac{\lvert\lvert e_3'\rvert\rvert}{\lvert\lvert e_3\rvert\rvert}\}$. Assume that $\frac{\lvert\lvert e_3\rvert\rvert}{\lvert\lvert e_3'\rvert\rvert}=\frac{\lvert\lvert e_1\rvert \rvert}{\lvert \lvert e_1'\rvert\rvert}$ or $\frac{\lvert\lvert e_3'\rvert\rvert}{\lvert\lvert e_3'\rvert\rvert}=\frac{\lvert\lvert e_1\rvert \rvert}{\lvert \lvert e_1'\rvert\rvert}$. Then it follows that
		
		$$\lvert\lvert e_1'\rvert\rvert \cdot \lvert\lvert e_3'\rvert\rvert\sin\theta_2'>\lvert\lvert e_1\rvert\rvert\cdot\lvert\lvert e_3\rvert\rvert\sin\theta_2,$$
		
		\noindent since $\theta<\theta_2'<\frac{\pi}{2}$. This is a contradiction since $T$ and $T'$ have the same area. 
\item
We know that 

$$1< \exp(m(T,T'))=\frac{\lvert\lvert e_2'\rvert \rvert}{\lvert \lvert e_2\rvert\rvert}\geq\max\{\frac{\lvert\lvert e_3\rvert\rvert}{\lvert\lvert e_3'\rvert\rvert},\frac{\lvert\lvert e_3'\rvert\rvert}{\lvert\lvert e_3\rvert\rvert}\}.$$

\noindent If $\frac{\lvert\lvert e_2'\rvert \rvert}{\lvert \lvert e_2\rvert\rvert}= \frac{\lvert\lvert e_3\rvert \rvert}{\lvert \lvert e_3'\rvert\rvert}$, then

	$$\lvert\lvert e_2\rvert\rvert\cdot\lvert\lvert e_3\rvert\rvert\sin\theta_1>\lvert\lvert e_2'\rvert\rvert\cdot\lvert\lvert e_3'\rvert\rvert\sin\theta_1',$$
	\noindent which is a contradiction. If $\frac{\lvert\lvert e_2'\rvert \rvert}{\lvert \lvert e_2\rvert\rvert}= \frac{\lvert\lvert e_3'\rvert \rvert}{\lvert \lvert e_3\rvert\rvert}$, then  we have 
		$$\lvert\lvert e_1'\rvert\rvert\lvert\lvert e_3'\rvert\rvert\sin\theta_2'=\lvert\lvert e_1\rvert\rvert\lvert\lvert e_3\rvert\rvert\sin\theta_2.$$
		
		\noindent 
	 Thus 
	 $$\frac{\lvert\lvert e_1 \lvert \lvert }{\lvert\lvert e_1'\rvert\rvert}=\frac{\lvert\lvert e_3' \lvert \lvert }{\lvert\lvert e_3\rvert\rvert}\frac{\sin \theta_2'}{\sin \theta_2}>\frac{\lvert\lvert e_3' \lvert \lvert }{\lvert\lvert e_3\rvert\rvert}=\frac{\lvert\lvert e_2' \lvert \lvert }{\lvert\lvert e_2\rvert\rvert}=\exp(m(T,T')),$$
	 
	 \noindent which is impossible. This completes the proof of the claim.
	
	\end{enumerate}

\begin{figure}
	\hspace*{-2cm} 
	\includegraphics[scale=1.2]{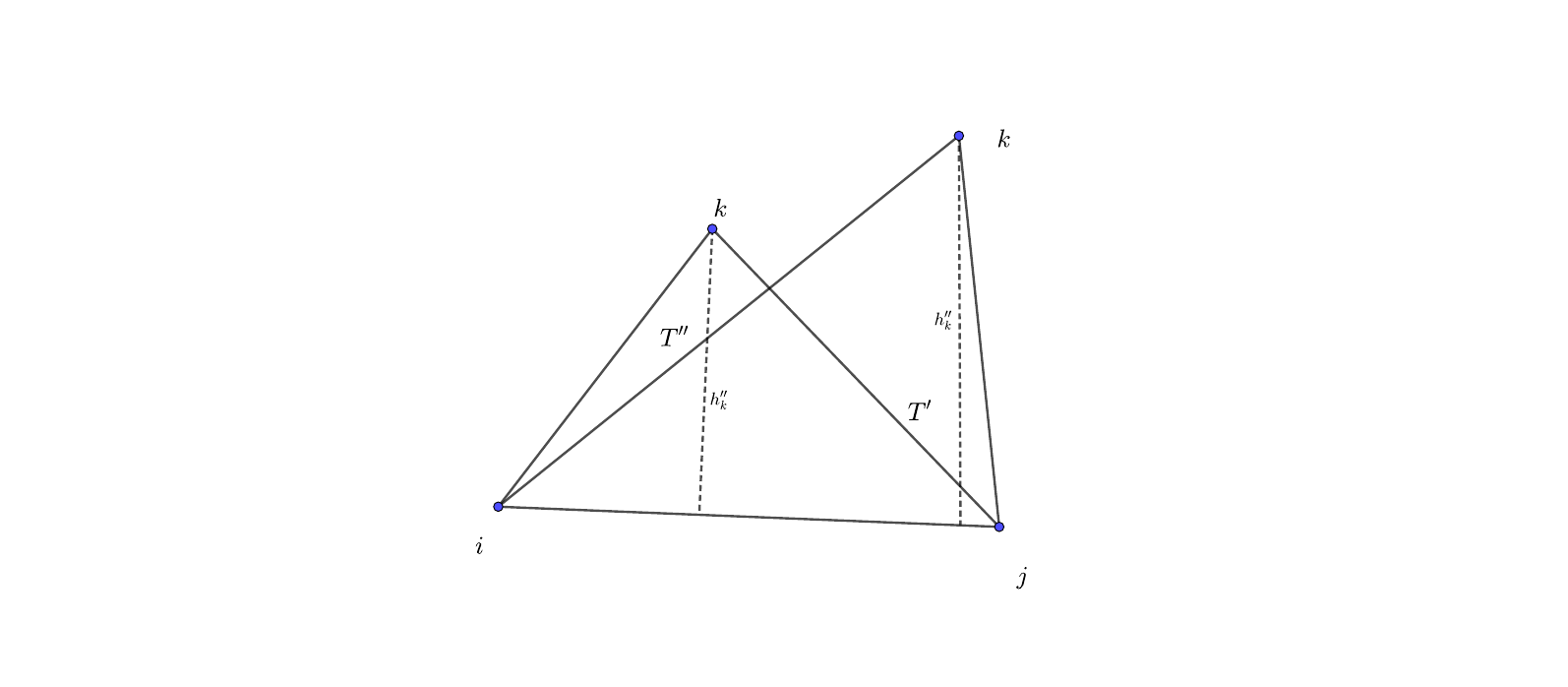}
	
	\caption{Two triangles sharing an edge but one is not contained in the other.}
	\label{crucial-2-fig}
\end{figure}
\end{proof}
\end{lemma}

\begin{remark}
	Let $T$ and $T'$ be two labeled triangles of equal area so that $\theta_i>\theta_i'$ and $\theta_j<\theta_j'$. Scale $T$ to get a new triangle $T''$ so that $\lvert\lvert e_k''\rvert\rvert=\lvert\lvert e_k'\rvert\rvert$. See Figure \ref{crucial-2-fig}.
	Then it follows from the discussion before Lemma \ref{crucial} that
	
	$$\max\{\frac{\lvert\lvert e_k\rvert\rvert}{\lvert\lvert e_k'\rvert\rvert},\frac{\lvert\lvert e_k'\rvert\rvert}{\lvert\lvert e_k\rvert\rvert}\}=\max\{\sqrt{\frac{\lvert\lvert h_3''\rvert\rvert}{\lvert\lvert h_3'\rvert\rvert}},\sqrt{\frac{\lvert\lvert h_3'\rvert\rvert}{\lvert\lvert h_3''\rvert\rvert}} \}$$
	
	\noindent Thus Lemma \ref{lem:two} implies that 
	
	$$\exp(m(T,T'))>\max\{\sqrt{\frac{\lvert\lvert h_3''\rvert\rvert}{\lvert\lvert h_3'\rvert\rvert}},\sqrt{\frac{\lvert\lvert h_3'\rvert\rvert}{\lvert\lvert h_3''\rvert\rvert}} \}.$$
\end{remark}

\section{The space of triangles}
\label{section-space-triangles}
In this section we will consider the set of isometry classes of triangles. First we define the notion of a metric. The definition we give is different from the usual definition of a metric since we drop the symmetry axiom.

\begin{definition}
	A metric on a set $X$ is a function $\eta:X\times X \to \R$ such that
	
\begin{itemize}
	\item 
	$\eta(x,x)=0$ for all $x\in X$,
	\item
	$\eta(x,y)>0$ if $x\neq y$,
	\item
	$\eta(x,y)+\eta(y,z)\geq \eta(x,z)$ for all $x,y,z \in X$.
	\end{itemize}
The pair $(X,d)$ or the set $X$ is called a metric space.
If $\eta(x,y)=\eta(y,x)$ for all $x,y \in X$, then the metric is called symmetric. Otherwise, the metric is called asymmetric.

\end{definition}

For a fixed $A>0$, let $\frak{T}_{A}$ be the set of equivalence classes of labeled triangles with area $A$, where two triangles are equivalent if there is an isometry between them which respects the labeling.  

We show that $L$ defines a metric on $\frak{T}_{A}$. If $T$ is a labeled triangle with area $A$, then we denote its equivalence class in $\frak{T}_{A}$ as $[T]$. It is clear that $L$ gives a well-defined function on $\frak{T}_{A}\times \frak{T}_{A}$. We denote this function by $L$ as well.

The proof of the following lemma is similar to that of Proposition 2.1 of \cite{thurston}.

\begin{lemma}
\label{thurston-type}
	Let $T$ and $T'$ be two labeled triangles of equal area. If $L(T,T')\leq 0$, then $L(T,T')=0$ and $T$ and $T'$ are isometric.
	\begin{proof}
		For any $\lambda \geq 0$ the set of $\lambda$-Lipschitz maps from $T$ to $T'$ which respect the labeling 
		is equicontinuous. Suppose that  $L(T,T')\leq 0$. Then we can pick a map $f: T \to T'$ which has minimum global Lipschitz constant $e^{L(T,T')}\leq 1$.
		
		 If $L(T,T')< 0$, we can assume that  $f$ is a homeomorphism and that its Lipschitz constant (which may be larger than $e^{L(T,T')}$) is $< 1$. But then we can cover the surface $T$ by a countable family of discs of different radii such that the interiors of these discs are disjoint and their complement has measure zero, such that $f$ maps one of these discs to a disc of strictly smaller radius. This is impossible since $T$ and $T'$ have the same area. Thus, $L(T,T')=0$. 
		 
		 Now taking a covering of the surface $T$ by a countable family of discs of different radii such that the interiors of these discs are disjoint and their complement has measure zero, we conclude,  since the Lipschitz constant of this map is 1, that each such disc is mapped by $f$ surjectively onto a disc of the same radius. Repeating the same argument with a covering of  $T$ by discs whose radii tend uniformly to zero, we see that $f$ is an isometry.
\end{proof}
\end{lemma}

\noindent Note that Lemma \ref{thurston-type} implies that for any $[T],[T']\in \frak{T}_{A}$ we have $L([T],[T'])\geq 0$ and $L([T],[T'])=0$ if and only if $[T]=[T']$.

\begin{lemma}
	\label{triangle-L}
	If $T,T',T''$ are three labeled triangles of the same area, then 
	$$L(T,T')+L(T',T'')\geq L(T,T'').$$
	
	\begin{proof}
		Let $f:T\to T'$ and $g:T'\to T'' $ be  label-preserving homeomorphisms. The assertion follows from the fact that
		$$L(g\circ f)\leq L(g)L(f).$$
		 
	\end{proof}
	
\end{lemma}

The following theorem is immediate from Lemma \ref{thurston-type} and Lemma \ref{triangle-L}.

\begin{theorem}
	The function $L$ is a metric on $\frak{T}_{A}$.
\end{theorem}

Clearly the function $m$ defined in Equation (\ref{formula-m}) induces a function on $\frak{T}_A\times\frak{T}_A$. We denote this function by $m$ as well.
Our next objective is to prove that $m$ is a symmetric metric on $\frak{T}_{A}$.

\begin{theorem}
	\label{m}
	$m$ is a symmetric metric on $\frak{T}_A$.
	
	\begin{proof}
		Let $[T], [T'] \in \frak{T}_A$. Then $T$ has edges of length $\lvert\lvert e_1 \rvert\lvert, \lvert\lvert e_2 \rvert\lvert, \lvert\lvert e_3 \rvert\lvert$ and $T'$ has edges of length $\lvert\lvert e'_1 \rvert\lvert, \lvert\lvert e'_2 \rvert\lvert, \lvert\lvert e'_3 \rvert\lvert$. Since $T$ and $T'$ have  same area, by Remark \ref{equal-area}, we have

	$$m(T,T')	
	=\log(\max\{\frac{\lvert\lvert e'_1 \rvert \rvert}{\lvert \lvert e_1 \rvert\rvert}, \frac{\lvert\lvert e'_2 \rvert \rvert}{\lvert \lvert e_2 \rvert\rvert},\frac{\lvert\lvert e'_3 \rvert \rvert}{\lvert \lvert e_3 \rvert\rvert},\frac{\lvert\lvert e_1 \rvert \rvert}{\lvert \lvert e'_1 \rvert\rvert},\frac{\lvert\lvert e_2 \rvert \rvert}{\lvert \lvert e'_2 \rvert\rvert},\frac{\lvert\lvert e_3 \rvert \rvert}{\lvert \lvert e'_3 \rvert\rvert}\}).$$
	
It is clear from this formula that $m$ separates points and $m$ is symmetric. Let $T''$
be another triangle with area $A$ having edges with length $\lvert\lvert e''_1 \rvert\lvert, \lvert\lvert e''_2 \rvert\lvert, \lvert\lvert e''_3 \rvert\lvert$. It follows from the  following inequality that $m$ satisfies the triangle inequality:

$$\max\{\frac{\lvert\lvert e'_1 \rvert \rvert}{\lvert \lvert e_1 \rvert\rvert}, \frac{\lvert\lvert e'_2 \rvert \rvert}{\lvert \lvert e_2 \rvert\rvert},\frac{\lvert\lvert e'_3 \rvert \rvert}{\lvert \lvert e_3 \rvert\rvert},\frac{\lvert\lvert e_1 \rvert \rvert}{\lvert \lvert e'_1 \rvert\rvert},\frac{\lvert\lvert e_2 \rvert \rvert}{\lvert \lvert e'_2 \rvert\rvert},\frac{\lvert\lvert e_3 \rvert \rvert}{\lvert \lvert e'_3 \rvert\rvert}\}$$

$$\times \max\{\frac{\lvert\lvert e''_1 \rvert \rvert}{\lvert \lvert e'_1 \rvert\rvert}, \frac{\lvert\lvert e''_2 \rvert \rvert}{\lvert \lvert e'_2 \rvert\rvert},\frac{\lvert\lvert e''_3 \rvert \rvert}{\lvert \lvert e'_3 \rvert\rvert},\frac{\lvert\lvert e'_1 \rvert \rvert}{\lvert \lvert e''_1 \rvert\rvert},\frac{\lvert\lvert a'_2 \rvert \rvert}{\lvert \lvert a''_2 \rvert\rvert},\frac{\lvert\lvert e'_3 \rvert \rvert}{\lvert \lvert e''_3 \rvert\rvert}\}$$

$$\geq \max\{\frac{\lvert\lvert e''_1 \rvert \rvert}{\lvert \lvert e_1 \rvert\rvert}, \frac{\lvert\lvert e''_2 \rvert \rvert}{\lvert \lvert e_2 \rvert\rvert},\frac{\lvert\lvert e''_3 \rvert \rvert}{\lvert \lvert e_3 \rvert\rvert},\frac{\lvert\lvert e_1 \rvert \rvert}{\lvert \lvert e''_1 \rvert\rvert},\frac{\lvert\lvert e_2 \rvert \rvert}{\lvert \lvert e''_2 \rvert\rvert},\frac{\lvert\lvert e_3 \rvert \rvert}{\lvert \lvert e''_3 \rvert\rvert}\}$$

\end{proof}

\end{theorem}

\noindent If there is no risk of confusion, we will denote  the equivalence class of a labeled triangle $T$ in $\frak{T}_A$ by $T$ as well.
\subsection{The space of acute triangles}
We denote the set of equivalence classes of acute triangles by $\frak{AT}$.  For each fixed $A>0$, let $\frak{AT}_A$ be the set of equivalence classes of acute triangles having area $A$. By Theorem \ref{acute-proof} the restrictions of $L$ and $m$ on $\frak{AT}_{A}$ give the same  metric.

\subsection{Two models for $\frak{AT}_{A}$}
In this section we introduce two models of $\frak{AT}_{A}$. First, since each labeled triangle is determined by the length of its edges, there is an injection $\frak{AT}_{A}\to \R^{*3}_+$
sending the class of a labeled triangle $T$ to $(\lvert \lvert e_1 \rvert \rvert,\lvert \lvert e_2 \rvert \rvert, \lvert \lvert e_3 \rvert \rvert)$, where $e_1,e_2,e_3$ are the edges of $T$. The image of this map is a 2-dimensional submanifold of $\R_+^{*3}$. Furthermore $m$ induces a metric in this manifold and we denote this metric by $m$ as well. Clearly, if $a=(a_1,a_2,a_3)$ and $a'=(a_1',a_2',a_3')$ are in this manifold, then

$$m(a,a')=\log(\max\{\frac{a_1}{a_1'},\frac{a_1'}{a_1},\frac{a_2}{a_2'},\frac{a_2'}{a_2},\frac{a_3}{a_3'},\frac{a_3'}{a_3}\}).$$

\noindent If there is no risk of confusion we will denote this submanifold by $\frak{AT}_{A}$ as well. We will call this model as the {\it edge model}.

Now we introduce the other model. We first note that each labelled triangle having area $A$ is determined by its angles. That is, the map $\frak{AT}_{A}\to \R^{*3}_+$ sending  a triangle $T$ to $(\theta_1,\theta_2,\theta_3)$ is injective,
where $\theta_1,\theta_2,\theta_3$ are angles at the vertices of $T$. Its image is the interior of a Euclidean  equilateral triangle in a hyperplane in $\R^{3}$. If there is no risk of confusion, we will also denote this image by $\frak{AT}_{A}$. We will denote by $(\theta_1,\theta_2,\theta_3)$ the triangle having angles $\theta_1,\theta_2,\theta_3$.  We will call this model the {\it angle model}.

\subsection{Topology of $\frak{T}_A$}

We consider the edge model for $\frak{T}_A$. It is equipped with two topologies: the one which comes from the Euclidean metric (the one induced from the ambient space) and the one which comes from $m$. We denote these topologies by $\mathcal{T}_{euc}$ and $\mathcal{T}_m$, respectively.

\begin{proposition}
	$\mathcal{T}_{euc}$ and $\mathcal{T}_m$ coincide.
	\begin{proof}
		It is enough to show that the identity map of $\frak{T}_A$ is a homeomorphism.
		This can be done by proving that a sequence $T_n=(a_n,b_n,c_n)\to (a,b,c)$ with respect to $m$ if and only if $T_n\to (a,b,c)$ with respect to $d_{euc}$. If $(a_n,b_n,c_n)\to (a,b,c)$ with respect to $d_{euc}$, then $\log a_n \to \log a$, $\log b_n \to \log b$ and $\log c_n \to \log c$
		as $n\to \infty$. Thus $(a_n,b_n,c_n)\to (a,b,c)$ with respect to $m$ as $n\to \infty$. If $(a_n,b_n,c_n)\to (a,b,c)$ with respect to $m$, then $\lvert \log a_n -\log a\rvert \to 0$, $\lvert \log b_n - \log b \rvert \to 0$ and $\lvert \log c_n - \log c \rvert \to 0$ as $n \to \infty$. Then $a_n\to a$, $b_n \to b$ and $c_n\to c$ as $n\to \infty$, but this means that $(a_n,b_n,c_n)\to (a,b,c)$ with respect to $d_{euc}$ as $n\to \infty$. 
	\end{proof}
\end{proposition}

\subsection{The Space of non-obtuse triangles}

We will prove that $\frak{T}_{A}$ is complete and introduce the space of non-obtuse triangles as the closure of  $\frak{AT}_{A}$ in $\frak{T}_{A}$. 

\begin{proposition}
	$(\frak{T}_A,m)$ is complete. 
	\begin{proof}
	We use the edge model of $\frak{T}_{A}$. Let $T_n=(a_n,b_n,c_n)$ be a Cauchy sequence of triangles with edge lengths $a_n,b_n,c_n$. 
	It follows that $\log a_n, \log b_n, \log c_n$ are Cauchy sequences. Then we have
	
$$\log a_n \to \log a, \log b_n \to \log b, \log c_n \to\log c.$$

\noindent $(a_n,b_n,c_n)\to (a,b,c)$ (with respect to  $m$). $(a,b,c)$ is a triangle with area $A$ since area is a continuous function of the length of edges of triangles, and $(a,b,c)$ satisfies strict triangle inequalities since the corresponding ``triangle" is non-degenerate. Indeed any degenerate triangle has zero area.
	\end{proof}
\end{proposition}

Furthermore, the set of non-obtuse triangles is a closed subset of $\frak{T}_A$. This follows from the fact that the angle functions are continuous. 
The set of acute triangles $\frak{AT}_A$ is an open subset of $\frak{T}_A$ and it  is easy to see that its closure in $\frak{T}_A$ is the set of non-obtuse triagles of area A, which is denoted by $\overline{\frak{AT}_A}$.

\subsection{$\frak{T}_A$ and $\frak{T}_{A'}$ are isometric}

For any two triangles
$T$ and $T'$, we have

$$m(\lambda T, \lambda T')= m(T,T'),$$

\noindent  It follows  that $\frak{T}_A$ and $\frak{T}_{\lambda^2A}$ are isometric by the map
sending the class of a triangle $T$ to the class of the triangle $\lambda T$. Similarly, for any $A,A'>0$, $\overline{\frak{AT}_A}$ and $\overline{\frak{AT}_{A'}}$ are isometric.

\section{ Geodesics in $\overline{\frak{AT}_A}$}
\label{section-special-geodesic}

Let $(X,d)$ be a metric space where $d$ is not necessarily symmetric. We use the following version of the notion of geodesic:

\begin{definition}
 A geodesic in $X$ is a map $h: I \to X$ where $I$ is an interval of $\R$ such that for any triple $x_1, x_2, x_3 \in I$ satisfying $x_1\leq x_2\leq x_3$, we have

$$d(h(x_1),h(x_3))=d(h(x_1),h(x_2))+d(h(x_2),h(x_3)).$$
We also require that $h$ is not constant on any 
subinterval of $I$.

\end{definition}
 Note that if $d$ is symmetric, a map obtained from $h$ by reversing the direction of a geodesic is also a geodesic. But this is not necessarily true for asymmetric metrics. Note also that if $d$ is asymmetric, one needs, in the above definition, to be careful about the order of the  arguments.

For $A>0$, let $T=(a_1,a_2,a_3)$ and $T'=(a_1',a_2',a_3')$
be two different elements in $\overline{\frak{AT}_A}$, where we use the edge model. In this section we prove that $T$ and $T'$ can be joined by a geodesic. In other words, we show that the metric space $(\overline{\frak{AT}_A},m)$ is geodesic.

As before, there is unique $i\in {1,2,3}$ such that one of the following holds:

\begin{enumerate}
	\item 
	$\theta_j\leq \theta'_j$ and $\theta_k\leq \theta'_k$,
	
\item
$\theta_j\geq \theta'_j$ and $\theta_k\geq \theta'_k$,
\end{enumerate}

\noindent where $j,k \in \{1,2,3\}$ and $j,k \neq i$. Without loss of generality, we  assume that $i=3$ and 
	
	$$\theta_1\geq \theta'_1 \ \text{and} \ \theta_2\geq \theta'_2,$$
	
\noindent and we shall find a geodesic from $T$ to $T'$. Since geodesics are reversible in a symmetric metric space, this argument should cover the case

	$$\theta_1\leq \theta'_1 \ \text{and} \ \theta_2\leq \theta'_2.$$
	
Consider the scaled triangles $\bar{T}=\frac{1}{a_3}T=(\frac{a_1}{a_3},\frac{a_2}{a_3},1)$ and $\bar{T}'=\frac{1}{a_3'}(\frac{a_1'}{a_3'},\frac{a_2'}{a_3'},1)$. The fact that they have an edge of equal length and the conditions on the angles imply that $\bar{T}'$ can be drawn inside $\bar{T}$. See Figure \ref{straight-line}. Let the set of vertices of $\bar{T'}$ be $\{\bar{v}'_1,\bar{v}'_2,\bar{v}'_3\}$ and the set of vertices of $\bar{T}$ be $\{\bar{v}_1,\bar{v}_2,\bar{v}_3\}$. Consider a family of triangles $\{\bar{T}_t\}$, $t\in [0,1]$, with angles $(\theta_1(t),\theta_2(t),\theta_3(t))$ having the following properties:

\begin{enumerate}
	\item 
For each $t$,	$\bar{T}_t$ is a triangle so that two of its vertices are $\bar{v_1}$ and $\bar{v_2}$.
	\item
	$\bar{T}_0= \bar{T}$ and $\bar{T}_1=\bar{T}'$.
	
\item
Each $\theta_i(t)$ is either a decreasing or increasing function.
\item
Each $\theta_i(t)$ is a continuous function.
\item
$(\theta_1(t),\theta_2(t),\theta_3(t))$ is not constant on any subinterval of $[0,1]$. 
\end{enumerate} 

\noindent See Figure \ref{straight-line} for an example of such a family.

If we scale each $\bar{T}_t$ by a factor $\lambda(t)$ so that the area of $\lambda(t)\bar{T}_t$ is $A$, then we get a family of triangles $T_t=(a_1(t),a_2(t),a_3(t))$
so that $T_0=(a_1,a_2,a_3)$ and $T_1=(a_1',a_2',a_3',)$. By Lemma \ref{crucial} this family has the following property:
\begin{equation}
	\label{bu-da}
\exp(m(T_{t_1},T_{t_2}))=\frac{a_3(t_2)}{a_3(t_1)}
\end{equation}
\noindent for any $t_1\leq t_2$, $t_1,t_2 \in [0,1]$. It follows that if $t_1\leq t_2 \leq t_3$ then

$$\exp(m)(T_{t_1},T_{t_3})=\frac{a_3(t_3)}{a_3(t_2)}\frac{a_3(t_2)}{a_3(t_1)}.$$
Thus,

$$m(T_{t_1},T_{t_3})= m(T_{t_1},T_{t_2})+m(T_{t_2},T_{t_3}).$$

\noindent This means that the 1-parameter family of triangles $T_t$ forms a geodesic. Therefore we proved the following.

\begin{figure}
	\hspace*{-2cm} 
	\includegraphics[scale=0.65]{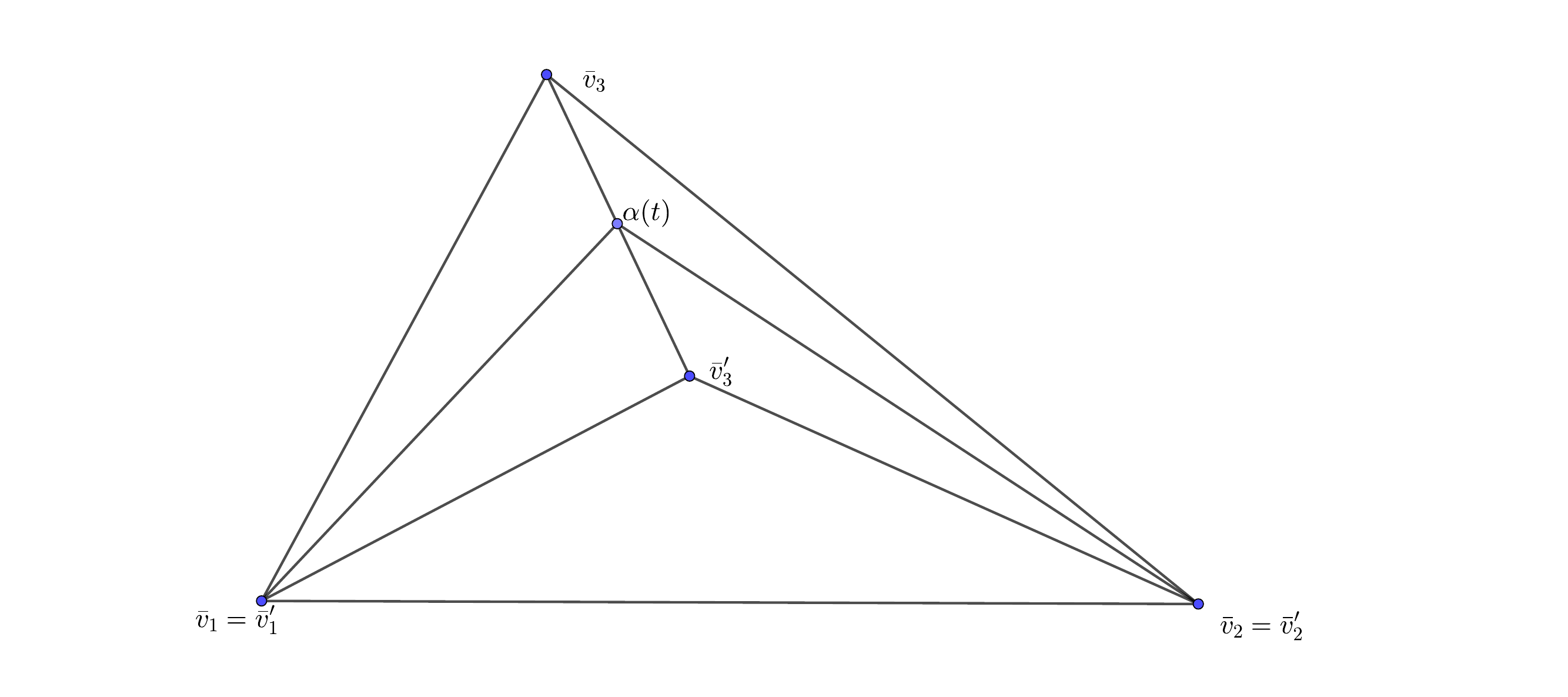}
	
	\caption{$\alpha(t)$ is a parametrization of the line segment between $\bar{v}_3$ and $\bar{v}_3'$. The family of triangles $\bar{T}_t=\Delta \bar{v}_1\bar{v}_2\alpha(t)$ is a geodesic in $\overline{\frak{AT}_A}$.}
	\label{straight-line}
\end{figure}

\begin{theorem} The space
	$(\overline{\frak{AT}_A},m)$  is geodesic, that is, any two points of $\overline{\frak{AT}}_A$ can be joined by a geodesic.
\end{theorem}

In fact, we also proved that any two distinct point in $\overline{\frak{AT}_A}$ can be joined by a special geodesic which is a straight line segment in the angle model of $\overline{\frak{AT}}_A$. To be more precise let $T=(\theta_1,\theta_2,\theta_3)$ and $T'=(\theta_1',\theta_2',\theta_3')$. Let 
$$\theta_i(t)=(1-t)\theta_i(t)+t\theta_i'(t), t\in [0,1], i \in \{1,2,3\}.$$

\noindent Then clearly $T_t=(\theta_1(t),\theta_2(t),\theta_3(t))$ is a geodesic joining $T$ and $T'$.

\begin{remark}
	If $T$ and $T'$ are two elements in $\overline{\frak{AT}_A}$ that have different angles at each vertex, then 
	there are uncountably many geodesics joining them up to parametrization. But if $T$ and $T'$ have equal angle at a vertex, then up to parametrization there is a unique geodesic joining them.
\end{remark}
Now we determine when a path in $\overline{\frak{AT}_A}$ is a geodesic.

\begin{theorem}
	\label{classification-geodesics}
Let $T_t=(\theta_1(t),\theta_2(t),\theta_3(t))$, $t\in [a,b]$ be a continuous family of triangles which is not constant on any subinterval of $[a,b]$. Then $T_t$ is a geodesic if and only 
if each $\theta_i(t)$ is increasing or decreasing.

\begin{proof}
	The above argument shows that if each $\theta_i(t)$ is decreasing or increasing, then $T_t$ is a geodesic. To prove the converse, we assume that there is a family $T_t$ which is a geodesic but not all $\theta_i(t)$ are monotone. We also assume that $\theta_1(a)\leq \theta_1(b)$, $\theta_2(a)\leq \theta_2(b)$ and $\theta_3(a)>\theta_3(b)$. It follows that one of the $\theta_i(t)$ is not monotone. Since it is not possible that both of $\theta_1(t)$ and $\theta_2(t)$ are increasing, but $\theta_3(t)$ is not monotone, we should only consider the cases where $\theta_1(t)$ or $\theta_2(t)$ are not increasing. By symmetry, we only need to assume that $\theta_1$ is not increasing.
	Since $\theta_1(a)\leq \theta_1(b)$, there 
	exists $t_1,t_2,t_3\in [a,b]$ such that 
	$t_1<t_2<t_3$ and $\theta_1(t_1)<\theta_1(t_2)>\theta_1(t_3)$.

	Consider the triangles $T_{t_1},T_{t_2}$ ans $T_{t_3}$.
		Scale these triangles so that the edges 
		opposite to the vertex with label 3 have the same length. See Figure \ref{uclu}.
		Let $\bar{T}_{t_i}$ be the scaled triangles. Let $ h_3(t_i)$ be the altitude of $\bar{T}_{t_i}$ from the vertex labeled by $3$. We have  
		\begin{align*}
m(T_{t_1},T_{t_3})& =  \log(\frac{\sqrt{\lvert\lvert h_3(t_3)\rvert\rvert}}{\sqrt{\lvert\lvert h_3(t_1)\rvert\rvert}}),\\ m(T_{t_1},T_{t_2}) 
 & \geq \log(\frac{\sqrt{\lvert\lvert h_3(t_2)\rvert\rvert}}{\sqrt{\lvert\lvert h_3(t_1)\rvert\rvert}}),\\ m(T_{t_2},T_{t_3}) 
 &> \log(\frac{\sqrt{\lvert\lvert h_3(t_3)\rvert\rvert}}{\sqrt{\lvert\lvert h_3(t_2)\rvert\rvert}}),
\end{align*}
(see Section \ref{a-remark-about}). Thus, $m(T_{t_1},T_{t_3})<m(T_{t_1},T_{t_2})+m(T_{t_2},T_{t_3})$. Hence the family $\{T_t\}$ does not form a geodesic.
	\end{proof}
\end{theorem}	
	
\begin{figure}
	\hspace*{-6cm} 
	\includegraphics[scale=0.6]{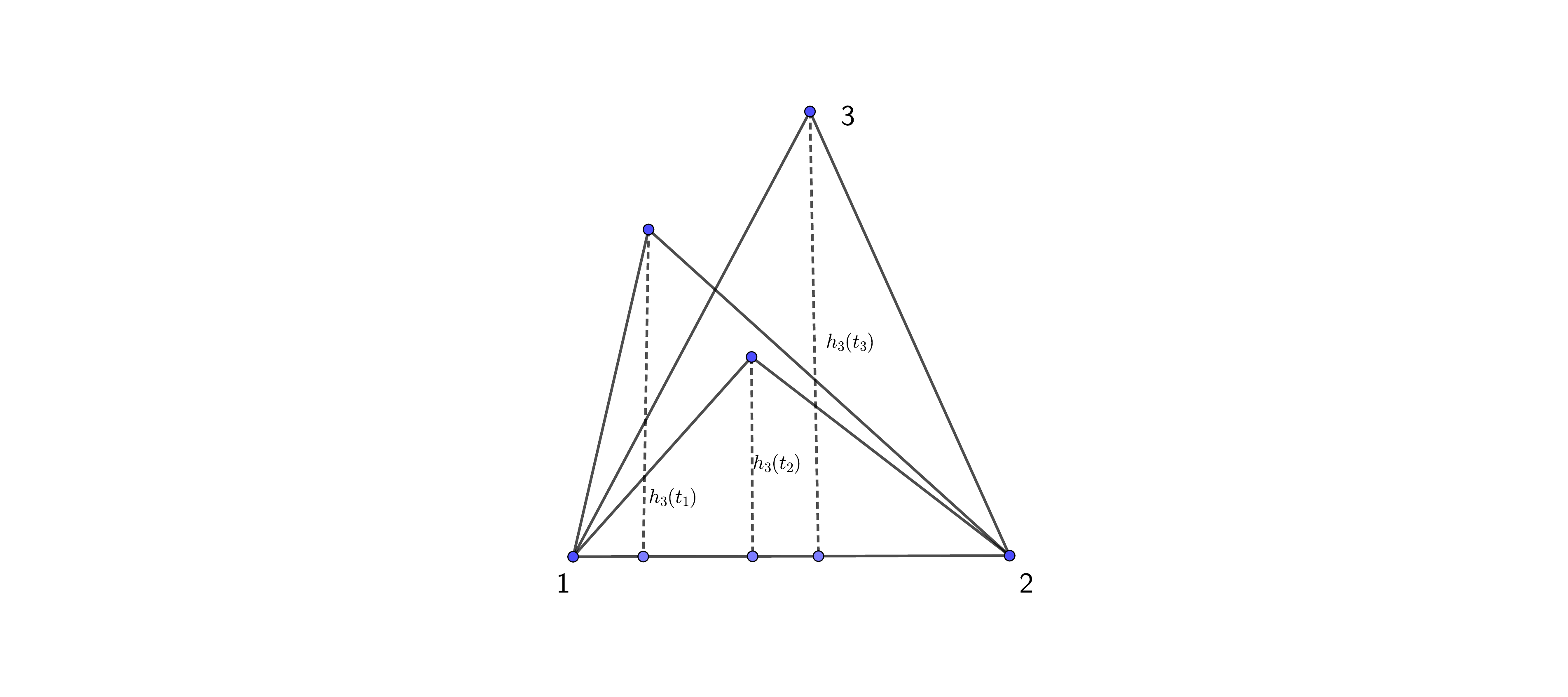}
	
	\caption{}
	\label{uclu}
\end{figure}

\section{The Finsler Structure on $\frak{AT}_A$}
\label{section-finsler}
Let us first recall the notion of a Finsler structure on a differentiable manifold. We work in the global Finsler setting adopted e.g. in \cite{PT-Finsler} (that is, without the tensor apparatus). We shall prove that the metric space $(\frak{AT}_{A},m)$, or, equivalently, $(\frak{AT}_A,L)$, is Finsler, that is, it is a length metric associated with a Finsler structure. We start with the definition of a {\it weak norm}.

\begin{definition}
	Let $V$ be a real vector space. A weak norm
	on $V$ is a function $V\to [0,\infty)$, $v \to \lvert\lvert v \rvert \rvert$ such that
	for every $v\in V$ the following properties hold for every $v$ and $w$ in $V$:
	\begin{enumerate}
		\item 
		$\lvert\lvert v \rvert\rvert=0$ if and only if $v=0$,
		\item
		$\lvert\lvert t v\rvert\rvert= t \lvert\lvert v \rvert \rvert $ for every $t>0$,
		\item
		$\lvert\lvert tv+(1-t)w\rvert\rvert\leq t\lvert\lvert v \rvert\rvert + (1-t)\lvert\lvert w \rvert\rvert$ for every $t\in [0,1]$.
	\end{enumerate}
\end{definition}

Let $M$ be a diffirentiable manifold and $TM$ be the tangent bundle of $M$. 

\begin{definition}
	A Finsler structure on $M$ is a function $F: TM \to [0,\infty)$ such that
	
	\begin{enumerate}
		\item 
	$F$ is continuous,
	\item
	for each $x\in M$, $F\lvert_{T_xM}$ is a weak norm.
	\end{enumerate}
\end{definition}

Let $F$ be a Finsler structure on a manifold $M$. For each $C^1$ curve $c: [a,b]\to M$, we define 

$$l(c)=l_F(c)=\int_{a}^{b}F(\dot{c}(t))dt.$$

\begin{definition}
	A metric $d$ on a differentiable manifold $M$ is called Finsler if it is the length metric associated to a Finsler structure, that is, if there exist a Finsler structure $F$
	on $M$ such that for every $x,y \in M$ we have
	
$$d(x,y)= \inf\{l_F(c)\}$$

\noindent where $c$ ranges over all piecewise $C^1$ curves such that $c(0)=x$ and $c(1)=y$. 
\end{definition}
Now we  show that the metric $m$ on $\frak{AT}_A$ is Finsler. Recall that we identified $\frak{AT}_A$ with the submanifold in $\R_+^{*3}$:

\begin{align*}
\frak{AT}_A= & \{(a_1,a_2,a_3) \in \R_+^{*3}: a_1+a_2-a_3>0,a_2+a_3-a_1>0,\\
& a_3+a_1-a_2 >0, \mathrm{Area}(a_1,a_2,a_3)=A\}.
\end{align*}

Let $F$ be the following function on $\frak{AT}_A$:

$$(a_1, a_2, a_3, v_1,v_2,v_3)\mapsto \max_i\{\frac{\lvert v_i \rvert}{a_i}\}.$$
Here, we have identified the tangent space of the manifold $\frak{AT}_A$ at a point $(a_1,a_2,a_3)$ with the Euclidean subspace of $ \R_+^{*3}$ spanned by the set $\frak{AT}_A$ itself, and $(v_1,v_2,v_3)$ are the coordinates of a tangent vector in the tangent space at the point $(a_1,a_2,a_3)$.

\noindent It is not difficult to see that $F$
is a Finsler structure on $\frak{AT}_A$. 

\begin{theorem}
	The metric $m$ on $\frak{AT}_{A}$ is Finsler.  More precisely, the metric $m$ 
	is the length metric associated with the Finsler structure $F$.
	\begin{proof}
		Let $a'=(a_1',a_2',a_3')$ and $a''=(a_1'',a_2'',a_3'')$ be in $\frak{AT}_{\frac{1}{2}}$. As in Section \ref{section-special-geodesic}, we may assume that 
		
		$$\exp(m(a',a''))= \frac{a_3''}{a_3'}$$
		
		\noindent and therefore there exists a $C^1$ function $g:[0,1]\to \frak{AT}_{A}$, $g(t)=(a_1(t),a_2(t),a_3(t))$ such that $g(0)=a',g(1)=a''$ and 
		
	$$\exp(m(c(t_1),c(t_2)))=\frac{a_3(t_2)}{a_3(t_1)} $$
for all  $t_1,t_2 \in [0,1], t_1\leq t_2$. Note that this implies that $a_3(t)$ is increasing.

We claim that 
\begin{equation}
	\label{budaonemli}
F(\dot{g}(t))= \frac{\lvert d(a_3(t))\rvert}{a_3(t)}=\frac{\dot{a}_3(t)}{a_3(t)}
\end{equation}
\noindent for each $t_0\in [0,1]$. Otherwise there exists $t_1\in[0,1]$ and $i\in\{1,2\}$ such that

$$\frac{\lvert \dot{a}_i(t_1)\rvert}{\lvert a_i(t_1)\rvert}>\frac{\dot{a}_3(t_1)}{a_3(t_1)}.$$

  Assume that $$\frac{ \dot{a}_i(t_1)}{ a_i(t_1)}>\frac{\dot{a}_3(t_1)}{a_3(t_1)}.$$
It follows that $\frac{d}{dt}(\log(\frac{a_i(t)}{a_3(t)}))>0$ at $t=t_1$ and we may assume that $t_1<1$. Therefore there is $t_2>t_1$ such that

$$\log(\frac{a_{i}(t_2)}{a_3(t_2)})>\log(\frac{a_i(t_1)}{a_3(t_1)}).$$
Hence we have
$$\frac{a_{i}(t_2)}{a_3(t_2)}>\frac{a_i(t_1)}{a_3(t_1)}, \ \text{or equivalently}\ \frac{a_{i}(t_2)}{a_i(t_1)}>\frac{a_3(t_2)}{a_3(t_1)},$$

\noindent This implies that $\exp(m(g(t_1),g(t_2)))\geq \frac{a_{i}(t_2)}{a_i(t_1)}>\frac{a_3(t_2)}{a_3(t_1)} $, which is a contradiction. In a similar manner we can show that 
$$\frac{-\dot{a}_i(t)}{a_i(t)}\leq \frac{\dot{a}_3(t)}{a_3(t)}.$$

Equation (\ref{budaonemli}) implies that 

$$l_F(g(t))=\int_{0}^1F(\dot{g}(t))=\int_{0}^1\frac{\dot{a}_3(t)}{a_3(t)}dt$$
$$=\log a_3(1)-\log a_3(0)=\log a_3'' -\log a_3'=m(a',a'').$$

Now we show that for any $C^1$ curve $c:[0,1]  \to \frak{AT}_{A}$, $c(t)=(a_1(t),a_2(t),a_3(t)).$ we have
	
	$$l_F(c)\geq m(a',a'')$$
	
We have 

$$L_f(c)\geq \int_{0}^1\frac{\lvert \dot{a}_i(t)\rvert}{a_i(t)}dt\geq \lvert \log a_i(1)-\log a_i(0)\rvert=\lvert \log a_i'' -\log a_i'\rvert. $$

\noindent Hence we have 

$$L_f(c)\geq  \max_i\{\lvert \log a_i'' - \log a_i' \rvert\}=m(a,a').$$

We conclude that 

$$m(a',a'')=\inf{l_F(c)},$$

\noindent where $c:[0,1]\to \frak{AT}_{A}$ is a $C^1$ curve such that $c(0)=a'$ 	and $c(1)=a''$.

	\end{proof}
\end{theorem}

\section{Symmetries of the space $\overline{\frak{AT}_A}$}
\label{section-symmetry}
In this section we study the isometry group $\text{Isom}(\overline{\frak{AT}_A})$ of $\overline{\frak{AT}_A}$ with respect to the metric $m$.

The space $\overline{\frak{AT}}_{A}$ has three boundary components, each of them corresponding to one angle of the triangle becoming right. If we consider the angle model and fix $i \in \{1,2,3\}$, then a boundary component is given as follows:

$$\{(\theta_1,\theta_2,\theta_3): 0<\theta_1,\theta_2,\theta_3\leq \frac{\pi}{2}, \theta_1+\theta_2+\theta_3=\pi, \theta_i=\frac{\pi}{2}\}.$$

This shows first that topologically $\overline{\frak{AT}}_{A}$ is a disc with three punctures on its boundary. A puncture may be regarded as a ``triangle" with angles $\frac{\pi}{2},\frac{\pi}{2},0$ and area $A$. Furthermore these boundary components are geodesics since any injective continuous map from $[0,1]$ to a boundary component is a geodesic. This can easily be deduced  from Theorem \ref{classification-geodesics}. We represent the boundary component for which $\theta_i=\frac{\pi}{2}$ by 
$\mathcal{B}_i$. Note that in the angle model, $\overline{\frak{AT}_A}$ is just a  Euclidean triangle with three punctures at its vertices.

It is not difficult to see that $\overline{\frak{AT}}_{A}$ is unbounded. Let us denote the equilateral triangle having area $\frac{1}{2}$ by $T_e$. Consider a sequence of isosceles triangles $T_n=(\theta_1(n),\theta_2(n),\theta_3(n))$ so that $\theta_2(n)=\theta_3(n)$ and $\lim_{n \to \infty}\theta_3(n)=\frac{\pi}{2}$. It requires a simple calculation to show that $m(T_e,T_n)\to \infty$ as $n \to \infty$. Now consider a sequence of triangles $T'_n=(\theta_1(n),\frac{\pi}{2},\theta_3(n))$, where $\theta_1(n)\to \frac{\pi}{2}$ as $n \to \infty$. Let us also consider the sequence $T''_n=(\frac{\pi}{2},\theta_1(n),\theta_3(n))$. It can be shown that $m(T'_n,T_n'')\to 0 $ as $n \to \infty$. See Remark \ref{near-puncture}. This means that $\overline{\frak{AT}}_{\frac{1}{2}}$ resembles an ideal triangle in the hyperbolic plane: it has three geodesic boundary components and any two of these components is a line which converges from each side to a puncture.

Now we consider the isometry group $\text{Isom}(\overline{\frak{AT}}_{A})$. The symmetric group $S_3=Sym\{1,2,3\}$ can be regarded as a subgroup of $\text{Isom}(\overline{\frak{AT}}_{A})$ as follows. Let $\sigma \in S_3$ and consider the length model for $\overline{\frak{AT}}_{A}$. We identify $\sigma$ with the map sending $(a_1,a_2,a_3)$ to $(a_{\sigma(1)},a_{\sigma(2)},a_{\sigma(3)})$. It is not difficult to see that this map is an isometry of $\overline{\frak{AT}}_{A}$. Thus we may consider $S_3$ as a subgroup of $\text{Isom}(\overline{\frak{AT}}_{A})$.  We wish to prove now that there are no other isometries, that is, $S_3=\overline{\frak{AT}}_{A}$. 

Let $\{i,j,k\}=\{1,2,3\}$. The 2-uple $\{T,T'\}$
is called an {\it$i$-pair} if $T \in \mathcal{B}_j$, $T'\in \mathcal{B}_k$ and $\theta_k=\theta'_j$.
Note that since $\theta_i=\theta'_i$, there is a unique geodesic between $T$ and $T'$ up to parametrization.

\begin{lemma}
\label{distance-boundary}
	If $\{T,T'\}$ and $\{T_1,T'_1\}$ are different   $i$-pairs, then $m(T,T')\neq m(T_1,T'_1)$.
	\begin{proof}
	Without loss of generality assume that $i=1, j=2, k=3$. Since $\overline{\frak{AT}_A}$  and $\overline{\frak{AT}_{\frac{1}{2}}}$ are isometric, we may suppose that $A=\frac{1}{2}$. This will make the calculations easier. Then in the edge model
	$$T=(a,\sqrt{a^2+\frac{1}{a^2}},\frac{1}{a}),  T'=(a, \frac{1}{a}, \sqrt{a^2+\frac{1}{a^2}}).$$
	It follows that
	$$m(T,T')= \log{\sqrt{a^2+\frac{1}{a^2}}}-\log \frac{1}{a}=\frac{1}{2}\log(1+a^4),$$
	
	\noindent which is an injective function of $a$. This proves the claim.
	\end{proof}
\end{lemma}

\begin{remark}
	\label{near-puncture}
	As $a\to 0$, $m(T,T')=\frac{1}{2}\log(1+a^4)\to 0$. This means that the distance between two boundary components is arbitrarily small near a puncture.
\end{remark}

\begin{theorem}
	$\text{Isom}(\overline{\frak{AT}_A})=S_3$.
	\begin{proof}
		It suffices to prove that an isometry that sends  $\mathcal{B}_i$ to itself for all $i \in \{1,2,3\}$ is the identity. Let $\sigma$ be such an isometry. Consider an $i-$pair $\{T,T'\}$. First of all observe that $\{\sigma(T),\sigma(T')\}$ is an $i$-pair. This is true since two elements in different boundary components are $l$-pairs for some $l \in \{1,2,3\}$ if and only if there is, up to parametrization, a unique geodesic between them. But to be joined by a unique geodesic is an isometry invariant. It follows that $m(T,T')=m(\sigma(T),\sigma(T'))$. Since $\sigma$ fixes boundary components, Lemma \ref{distance-boundary} implies that $T=\sigma(T)$ and $T'=\sigma(T')$. Therefore $\sigma$ gives us an isometry of the geodesic between $T$ and $T'$. 
		But if we restrict $m$ to this   geodesic, we get the usual metric on some interval in the real line. Since an isometry of an interval fixing its endpoint is trivial, it follows that the restriction of $\sigma$ to such a geodesic is the identity. Since any point in $\overline{\frak{AT}}_A$ lies in such a geodesic, it follows that $\sigma$ is the identity map. 
	\end{proof}
\end{theorem}

\begin{figure}
	\hspace*{-1 cm} 
	\includegraphics[scale=0.85]{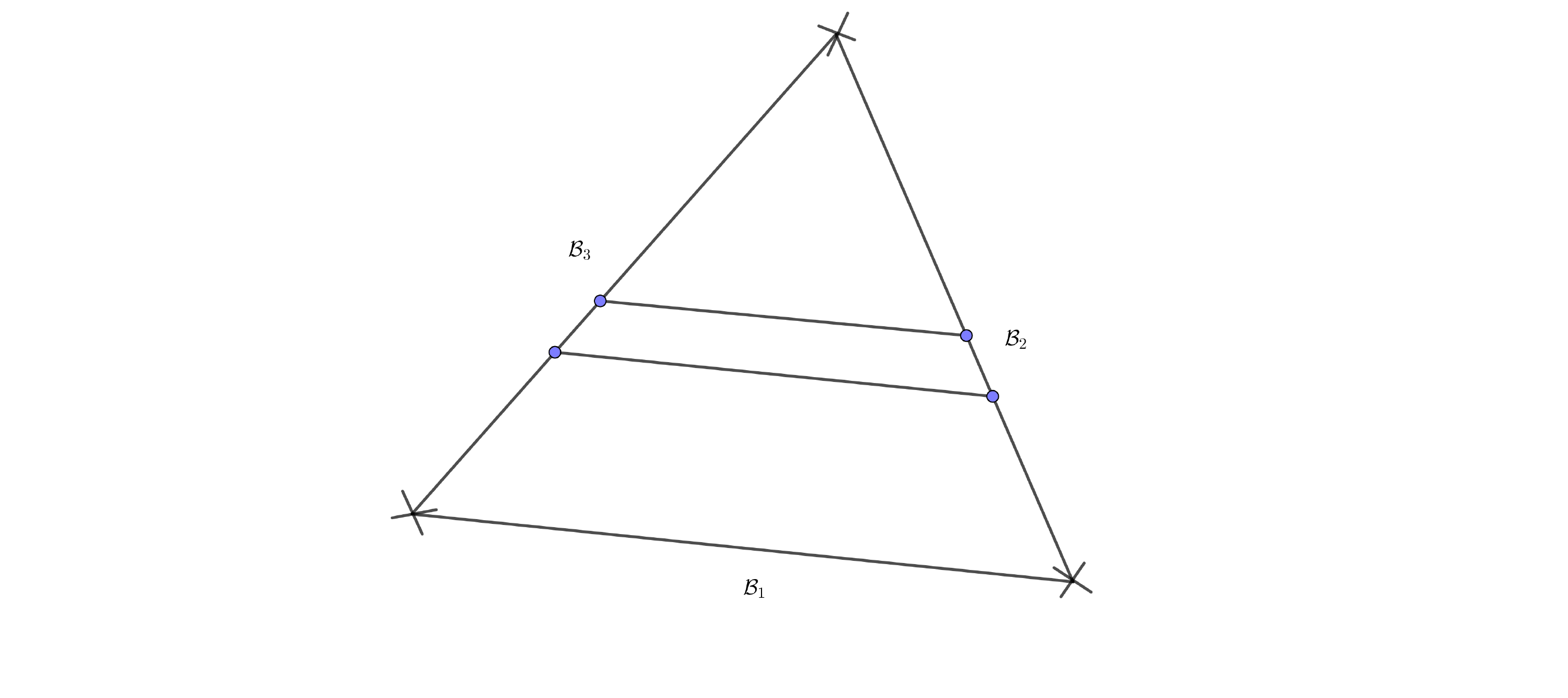}
	
	\caption{The angle model}
	\label{angle-model}
\end{figure}

In Figure \ref{angle-model}, we give the angle model together with some $1$-pairs and geodesics between them. Note that the geodesic between 
two $i$-pairs is a straight line segment  which is parallel to the boundary component $\mathcal{B}_i$. Consider the function from $\overline{\frak{AT}_A} \to \R_+^*$ which sends 
a triangle to its angle at the $i$-th vertex. Then it follows that the inverse image of any point in $\R^*_+$ is a geodesic segment between two $i$-pairs.

	\bigskip

\noindent {\bf Acknowledgements}
The first named author is financially supported by T\"{U}B\.{I}TAK.

\end{document}